\documentclass[12pt]{article}
\usepackage{amsmath}
\usepackage{amsthm}
\usepackage{enumerate}
\usepackage{amssymb}

\usepackage{verbatim}
\usepackage{url}
\usepackage[latin1]{inputenc}

\usepackage{caption}
\usepackage{graphicx,color}
\DeclareMathOperator{\Var}{Var}

\def\be{\begin{equation}}
\def\ee{\end{equation}}
\def\bea{\begin{equation*}}
\def\eea{\end{equation*}}
\def\begs{\begin{split}}
\def\ends{\end{split}}

\newtheorem{thm}{Theorem}[section]
\newtheorem{lma}[thm]{Lemma}

\newtheorem{prop}[thm]{Proposition}

\newtheorem{df}[thm]{Definition}

\theoremstyle{remark}
\newtheorem{preremark}[thm]{Remark}
\newtheorem{preex}[thm]{Example}

\newenvironment{remark}{\begin{preremark}}{\qed\end{preremark}}

\numberwithin{equation}{section}

\title{Sublinear variance in first-passage percolation for general distributions}
\author{Michael Damron\thanks{The research of M. D. is supported by NSF grant DMS-0901534.} \and Jack Hanson\thanks{The research of J. H. is supported by an NSF graduate research fellowship.} \and Philippe Sosoe\thanks{The research of P. S. is supported by an NSERC postgraduate fellowship.}}
\begin{document}

\maketitle

\abstract{We prove that the variance of the passage time from the origin to a point $x$ in first-passage percolation on $\mathbb{Z}^d$ is sublinear in the distance to $x$ when $d\geq 2$, obeying the bound $C\|x\|/\log \|x\|$, under minimal assumptions on the edge-weight distribution. The proof applies equally to absolutely continuous, discrete and singular continuous distributions and mixtures thereof, and requires only $2+ \log $ moments. The main result extends work of Benjamini-Kalai-Schramm \cite{BKS} and Benaim-Rossignol \cite{benaimrossignol}.}

%\tableofcontents

\section{Introduction}

\subsection{Background}
In addition to its rich stochastic geometric structure, first-passage percolation on $\mathbb{Z}^d$ provides a model for the study of fluctuations of a non-linear function of a large number of independent random variables. For recent surveys, see \cite{BStahn, GK, Howard}

In this paper, we are concerned with the variance of the passage time $\tau(0,x)$ from $0$ to $x\in \mathbb{Z}^d$. The passage time is the random variable defined as
\begin{equation}\label{eq: tau_def}
\tau(0,x) =\inf_{\gamma: 0\rightarrow x}\sum_{e\in \gamma} t_e\ ,
\end{equation}
where the infimum is taken over all lattices paths $\gamma=(v_0=0,e_0,v_1,\ldots,e_N,v_N=x)$ joining $0$ to $x$. The collection $(t_e)_{e\in \mathcal{E}^d}$ consists of nonnegative independent random variables with common distribution $\mu$ and $\mathcal{E}^d$ is the set of nearest-neighbor edges.

When $d=1$, \eqref{eq: tau_def} is simply a sum over i.i.d. random variables for each $x$, and the variance of $\tau(0,x)$ is of order $\|x\|_1$. In contrast, when $d\ge 2$, \eqref{eq: tau_def} is a mininum over a collection of correlated sums of i.i.d. random variables. This correlation structure has led physicists to conjecture a sublinear scaling of the form $\|x\|_1^\alpha$, $\alpha<1$ for the fluctuations. In the case $d=2$, the model is expected to have \emph{KPZ scaling} \cite{kpz}, with $\alpha=\frac{2}{3}$, and the recentered passage time approximately follows the Tracy-Widom distribution. Except for K. Johansson's work \cite{johansson} on a related exactly solvable model, there has been little success in rigorously confirming these predictions.

In \cite{kestenspeed}, H. Kesten showed that the variance of $\tau(0,x)$ is at most linear in the distance of $x$ to the origin:
\[\Var \tau(0,x)\le C\|x\|_1,\]
for some constant $C$. Kesten also showed that if $\mu$ has exponential moments:
\begin{equation}
\int e^{\delta x}\,\mu(\mathrm{d}x)<\infty \text{ for some } \delta>0\ ,\label{eqn: exp-moments}
\end{equation}
then the passage time is exponentially concentrated around its mean:
\begin{equation}\label{eqn: kesten}
\mathbb{P}(|\tau(0,x)-\mathbb{E}\tau(0,x)|\ge \lambda \sqrt{\|x\|_1})\le Ce^{-c\lambda},
\end{equation}
for $\lambda \le C\|x\|_1$.
M. Talagrand improved this result to Gausssian concentration on the scale $\sqrt{\|x\|_1}$: see \cite[Proposition~8.3]{talagrand}. These results have been used to derive concentration of the mean of the passage time around the ``time constant.'' Some relevant papers include \cite{alexander, rhee, zhang}. In the other direction, lower bounds have been given for the variance of the passage time, but the strongest results are dimension-dependent; see \cite{AD11, kestenspeed, newman, zhang2}.

In a remarkable paper \cite{BKS}, I. Benjamini, G. Kalai, and O. Schramm used an inequality due to Talagrand \cite{talagrand-russo} to prove that if the edge-weight distribution is uniform on a set of two positive values, the variance is sublinear in the distance:
\[\Var \tau(0,x) \le C(a,b)\frac{\|x\|_1}{\log\|x\|_1},~ d \geq 2\]
for $0<a<b$ and $\mathbb{P}(t_e=a)=\mathbb{P}(t_e = b) =\frac{1}{2}.$
M. Benaim and R. Rossignol \cite{benaimrossignol} introduced their ``modified Poincar\'e inequality,'' itself based on an inequality of D. Falik and A. Samorodnitsky (a corresponding inequality appears in Rossignol \cite[Equations~(11)-(14)]{Rossignol}), to extend the variance estimate to a class of continuous distributions which they termed ``nearly gamma.'' Nearly gamma distributions satisfy an entropy bound analogous to the logarithmic Sobolev inequality for the gamma distribution, which explains their name; for a nearly gamma $\mu$ and, for simplicity, $f$ smooth,
\begin{equation}\label{eqn: nearlygamma}
Ent_\mu f^2 := \int f^2(x) \log \frac{f^2(x)}{\mathbb{E}_\mu f^2}\,\mu(\mathrm{d}x) \le C\int \left(\sqrt{x}f'(x)\right)^2\,\mu(\mathrm{d}x).
\end{equation}
Benaim and Rossignol also show exponential concentration at scale $\sqrt{\|x\|_1/\log \|x\|_1}$ for nearly gamma distributions with exponential moments: if $\mu$ satisfies \eqref{eqn: nearlygamma} and \eqref{eqn: exp-moments},
then
\begin{equation}\label{eqn: concentration}
 \mathbb{P}_\mu(|\tau(0,x)-\mathbb{E}_\mu \tau(0,x)| \ge \lambda \sqrt{\|x\|_1/\log \|x\|_1}) \le Ce^{-c\lambda}.
\end{equation}

The nearly gamma condition excludes many natural distributions, including all power law distributions and distributions with infinite support which decay too quickly, mixtures of continuous and discrete distributions, singular continuous distributions, and continuous distributions with disconnected support, or whose density has zeros on its support.

\subsection{Main result}

The purpose of the present work is to extend the sublinear variance results mentioned above to general distributions with $2+\log$ moments. We make two assumptions:
\begin{equation}\label{eqn: 2log}
\int x^2(\log x)_+~\mu(\text{d}x) < \infty\ ,
\end{equation}
\begin{equation}\label{eqn: geodesics}
\mu(\{0\})<p_c(d)\ ,
\end{equation}
where $p_c(d)$ is the critical parameter for bond percolation on $\mathbb{Z}^d$. 

Our main result is the following:
\begin{thm}\label{thm: sub linear} Let $\mu$ be a Borel probability measure supported on $[0,\infty)$ satisfying \eqref{eqn: 2log} and \eqref{eqn: geodesics}. In i.i.d. first-passage percolation on $(\mathbb{Z}^d,\mathcal{E}^d)$, $d\ge 2$, with edge-weight distribution $\mu$, there exists a constant $C=C(\mu,d)$ such that 
\[
\Var \tau(0,x) \le C\frac{\|x\|_1}{\log \|x\|_1} \text{ for all } x \in \mathbb{Z}^d\ .
\]
\end{thm}

\begin{remark}
When \eqref{eqn: geodesics} fails, the passage time is known to be bounded by $C\|x\|_1^\epsilon$ for any $\epsilon$. See \cite{chayes,zhang3} for more details.
\end{remark}

\begin{remark}
The moment condition $\mathbb{E}t_e^2(\log t_e)_+<\infty$ may be able to be weakened, perhaps as low as $\mathbb{E}t_e^{(2/d)+a}<\infty$ for some $a>0$ by tensorizing entropy over small blocks, as in \cite[Lemma~2.6]{DK}. The main reason is that, due to \cite[Lemma~3.1]{coxdurrett}, $\Var \tau(x,y) < \infty$ for all $x,y$ under the condition $\mathbb{E} t_e^{(1/d)+a} < \infty$ for some $a>0$.
\end{remark}

Our method of proof may be of independent interest. Following \cite{benaimrossignol}, we use a martingale difference decomposition and the inequality of Falik and Samorodnitsky to control the variance of an averaged version of $\tau(0,x)$ by the entropy times a $1/\log \|x\|_1$ factor. Instead of representing the measure $\mu$ as the pushfoward of a Gaussian by an invertible transformation and using the Gaussian logarithmic Sobolev inequality, we represent $\mu$ as the image of an infinite sequence of uniform Bernoulli variables, and use A. Bonami and L. Gross's two-point entropy inequality \cite{bonami, gross} (the ``discrete log-Sobolev inequality'') to control the entropy. A central part of the argument is then to estimate the discrete derivatives of $\tau(0,x)$ with respect to variations of the Bernoulli variables.

\subsection{Outline of the paper}

The plan of the paper is as follows: in Section~\ref{sec: entropy}, we review some basic properties of the entropy functional with respect to a probability measure, and present the inequality of Falik and Samorodnitsky which we will use. In Section~\ref{sec: bernoulli}, we apply this inequality to first-passage percolation, using the martingale decomposition introduced in \cite{benaimrossignol}. We then briefly explain Benaim and Rossignol's approach based on the Gaussian log-Sobolev inequality in Section~\ref{sec: BR}, and show that a modification of their method using positive association already allows one to deal with a larger class of continuous distributions than the ones handled in \cite{benaimrossignol}. The purpose of Section~\ref{sec: BR} is only to clarify the role of conditions appearing in \cite{benaimrossignol}. This section is independent of the derivation of our main result.

In Section~\ref{sec: lower}, we provide a lower bound for the quantity $\sum_{k=1}^\infty (\mathbb{E} |V_k|)^2$ appearing in the variance bound, which will give the logarithmic factor in the final inequality. Next, in Section~\ref{sec: derivatives} we represent the passage time variables through a Bernoulli encoding and, after applying Bonami's inequality, bound a sum of discrete derivatives with the help of estimates on greedy lattice animals.

%In a final Section~\ref{sec: entropybds}, we remark that our entropy bound can be used to obtain inequalities of log-Sobolev type for some distributions with constants of the right order of magnitude. We plan to give further applications of these observations in forthcoming work \cite{DHS13}.

\subsection{Notation and preliminary results}

We will work on the space $\Omega = [0,\infty)^{\mathcal{E}^d}$ and let $\mu$ be a Borel probability measure on $[0,\infty)$. The product measure $\prod_{e \in \mathcal{E}^d} \mu$ will be denoted by $\mathbb{P}$. A realization of passage times (edge-weights) $\omega \in \Omega$ will be written as $\omega = (t_e)$ with point-to-point passage time $\tau(x,y)$ given by \eqref{eq: tau_def}. Throughout the paper, the letter $I$ will refer to the infimum of the support of $\mu$: writing 
\begin{equation}\label{eq: F_def}
F(x) = \mu((-\infty,x])
\end{equation}
for the distribution function of $\mu$, set
\begin{equation}\label{eq: I_def}
I = \inf\{x : F(x) > 0\}\ .
\end{equation}

A fundamental object in first-passage percolation is a geodesic, and we spend some time here giving some basic properties of geodesics. Any path $\gamma$ from $x$ to $y$ with passage time $\tau(\gamma) = \sum_{e \in \gamma} t_e$ satisfying $\tau(\gamma) = \tau(x,y)$ will be called a geodesic from $x$ to $y$. From the shape theorem of Cox-Durrett \cite{coxdurrett} and the fact that under \eqref{eqn: geodesics}, the limiting shape for the model is bounded \cite[Theorem~6.1]{kesten}, assumptions \eqref{eqn: 2log} and \eqref{eqn: geodesics} ensure the existence of geodesics:
\begin{equation}\label{eq: exist_geodesics}
\mathbb{P}( \text{for all }x,y \in \mathbb{Z}^d \text{ there exists a geodesic from }x \text{ to } y) = 1\ .
\end{equation}
There is almost surely a unique geodesic between $x$ and $y$ if and only if $\mu$ is continuous, so this need not be true in general. For any $x,y \in \mathbb{Z}^d$ we then use the notation
\begin{equation}\label{eq: geo_def}
Geo(x,y) = \{e \in \mathcal{E}^d : e \in \gamma \text{ for all geodesics } \gamma \text{ from } x \text{ to } y\}\ .
\end{equation}

Central to the current proofs of variance bounds for the passage time are estimates on the length of geodesics. The key theorem is due to Kesten \cite[(2.25)]{kestenspeed} and is listed below. We will need to derive two generalizations of this result. The first is Lemma~\ref{lem: exp_intersections} and concerns the number of intersections of $Geo(0,x)$ with arbitrary edge sets. The second, Theorem~\ref{thm: low_density}, gives a bound on the number of edges of $Geo(0,x)$ whose weight lies in a given Borel set.
\begin{thm}[Kesten]\label{thm: kesten_length}
Assume $\mathbb{E}t_e<\infty$ and \eqref{eqn: geodesics}. There exists $\mathbf{C}_1$ such that for all $x$,
\[
\mathbb{E}\#Geo(0,x) \leq \mathbf{C}_1\|x\|_1\ .
\]
\end{thm}

The second tool we shall need is \cite[Propsition~5.8]{kesten} and shows that under assumption \eqref{eqn: geodesics}, it is unlikely that long paths have small passage time. 
\begin{thm}[Kesten]\label{thm: kesten_exp}
Assuming \eqref{eqn: geodesics}, there exist constants $a, \mathbf{C}_2>0$ such that for all $n \in \mathbb{N}$,
\[
\mathbb{P}\bigg(\exists \text{ self-avoiding } \gamma \text{ starting at }0 \text{ with } \#\gamma \geq n \text{ but with } \tau(\gamma) < an \bigg) \leq \exp(-\mathbf{C}_2 n)\ .
\]
\end{thm}

\subsection{Proof sketch}
\noindent
{\bf The setup.} Our argument begins with the setup of Benaim and Rossignol: to bound the variance, we use the inequality of Falik-Samorodnitsky. That is, if $T = \tau(0,x)$ is the passage time, then we enumerate the edges of the lattice as $\{e_1, e_2, \ldots\}$ and perform a martingale decomposition
\[
T - \mathbb{E}T = \sum_{k=1}^\infty V_k\ ,
\]
where $V_k = \mathbb{E}[T \mid \mathcal{F}_k] - \mathbb{E}[T \mid \mathcal{F}_{k-1}]$ and $\mathcal{F}_k$ is the sigma-algebra generated by the edge weights $t_{e_1}, \ldots, t_{e_k}$. Then one has
\[
\Var T~\log \left[ \frac{\Var T}{\sum_{k=1}^\infty (\mathbb{E}|V_k|)^2} \right] \leq \sum_{k=1}^\infty Ent(V_k^2)\ .
\]
(See Lemma~\ref{lem: lower_bound}.) If $\Var T \leq \|x\|^{7/8}$, then the required bound holds; otherwise, one has $\Var T \geq \|x\|^{7/8}$ and the bound is
\[
\Var T~\log \left[ \frac{\|x\|^{7/8}}{\sum_{k=1}^\infty (\mathbb{E}|V_k|)^2} \right] \leq \sum_{k=1}^\infty Ent(V_k^2)\ .
\]
By working with an averaged version $F_m$ of $T$ (similar to that used in \cite{BKS}, but a different definition that simplifies the analysis and requires a new argument) one can ensure that the sum in the denominator on the left is at most order $\|x\|^{3/4}$. (See Proposition~\ref{prop: 5.3}.) Thus we begin our analysis with
\begin{equation}\label{eq: setup}
\Var T \leq \frac{C}{\log \|x\|} \sum_{k=1}^\infty Ent(V_k^2)\ .
\end{equation}

\medskip
\noindent
{\bf Step 1.} Bernoulli encoding. If one knows a log-Sobolev inequality (LSI) of the form $Ent~f^2 \leq C \mathbb{E}\|\nabla f\|_2^2$, then the argument of Benaim-Rossignol would give $\sum_{k=1}^\infty Ent(V_k^2) \leq C\mathbb{E}\|\nabla T\|_2^2$ and the method of Kesten can give an upper bound on this term by $C\|x\|_1$. Combining with \eqref{eq: setup} gives the sub-linear variance bound.

Unfortunately very few distributions satisfy a LSI of the above type. Benaim-Rossignol deal with this by exhibiting certain edge-weight distributions (those in the ``nearly gamma'' class) as images of Gaussian random variables and using the Gaussian LSI. This does not work for all distributions, so our main idea is to encode general edge-weights using infinite sequences of Bernoulli variables and use the Bernoulli (two-point) LSI.

For simplicity, assume that the edge-weights $t_e$ are uniformly distributed on $[0,1]$, so that we can encode their values using the binary expansion and i.i.d. Bernoulli$(1/2)$ sequences 
\[
t_e = \sum_{i=1}^\infty \omega_{e,i} 2^{-i}, \text{ where } (\omega_{e,1}, \omega_{e,2}, \ldots) \text{ is i.i.d. Bernoulli}(1/2)\ .
\]
(For general distributions, we compose with the right-continuous inverse of the distribution function of $t_e$.) Then using the Bernoulli LSI and the argument of Benaim-Rossignol,
\begin{equation}\label{eq: LSI}
\sum_{k=1}^\infty Ent(V_k^2) \leq 2\sum_{k=1}^\infty \sum_{i=1}^\infty \mathbb{E}(\Delta_{e_k,i} T)^2\ ,
\end{equation}
where $\Delta_{e_k,i}$ is the discrete derivative of $T$ relative to flipping the $i$-th bit in the binary expansion of $t_{e_k}$. This is done in Lemma~\ref{eqn: entropybydelta}.

\medskip
\noindent
{\bf Step 2.} The bulk of the paper is devoted to bounding these discrete derivatives: giving the inequality
\[
\sum_{k=1}^\infty \sum_{i=1}^\infty \mathbb{E}(\Delta_{e_k,i} T)^2 \leq C \|x\|_1\ .
\]
This is not a priori clear because flipping bits in the binary expansion can have a large change on $t_e$ if there are gaps in the support of the edge-weight distribution. We deal with this by considering this influence ``on average." That is, letting $\mathbb{E}_{< i}$ be the expectation over the binary variables $\omega_{e_k,1}, \ldots, \omega_{e_k,i-1}$, one has
\[
\mathbb{E}_{< i}(\Delta_{e_k,i}T)^2 = \frac{1}{2^{i-1}} \sum_{\sigma \in \{0,1\}^{i-1}} (T(\sigma,1)-T(\sigma,0))^2\ ,
\]
where we have indicated dependence of $T$ only on the first $i$ binary variables. Because the weights are bounded in $[0,1]$, the differences above are at most 1 (and nonzero only when $e_k$ is in a geodesic from $0$ to $x$ for some value of $t_e$), so we can telescope them, obtaining the upper bound
\[
\mathbf{1}_{\{e_k \in Geo(0,x) \text{ for some value of }t_{e_k}\}} \frac{1}{2^{i-1}} \sum_{\sigma \in \{0,1\}^{i-1}} (T(\sigma,1) - T(\sigma,0)) \leq \frac{1}{2^{i-1}} \mathbf{1}_{\{e_k \in Geo(0,x) \text{ for some value of } t_{e_k}\}}\ .
\]
Pretending for the moment that the indicator is actually of the event that $e_k \in Geo(0,x)$, we can sum over $i$ to give the bound $2 \mathbf{1}_{\{e_k \in Geo(0,x)\}}$, and sum over $k$, using Theorem~\ref{thm: kesten_length}, to obtain
\[
\sum_{k=1}^\infty Ent(V_k^2) \leq C \sum_k \mathbb{P}(e_k \in Geo(0,x)) \leq C\|x\|_1\ .
\]

\medskip
\noindent
{\bf Step 3.} General case. We are not using only uniform $[0,1]$ edge weights, so several complications arise, due both to large edge-weights and to edge-weights near the infimum of the support. The first problem forces the moment condition $\mathbb{E}t_e^2(\log t_e)_+<\infty$ and the second is related to the change from $\mathbf{1}_{\{e \in Geo(0,x) \text{ for some } t_e\}}$ to $\mathbf{1}_{\{e \in Geo(0,x)\}}$. However, careful bounding (for example, keeping track of the value $D_{z,e}$ of the edge-weight above which the edge leaves the geodesic -- see Lemma~\ref{lem: deriv}) leads to the inequality in Proposition~\ref{prop: intermediate}:
\begin{equation}\label{eq: last_entropy_bound}
\sum_{k=1}^\infty Ent(V_k^2) \leq C \mathbb{E} \sum_e (1-\log F(t_e)) \mathbf{1}_{\{e \in Geo(0,x)\}}\ ,
\end{equation}
where $F(t_e)$ is the distribution function of the weight $t_e$. Note that this is large when $t_e$ is near its infimum. In a sense, \eqref{eq: last_entropy_bound} is our version of an LSI, with the penalties due to the fact that we do not have a traditional LSI.

For certain distributions, we can bound $(1-\log F(t_e)) \leq C$ and sum as above. In particular, this is possible when there is an atom at the infimum of the support. But for general distributions, we must analyze the number of edges in the geodesic which have weight near the infimum. For this we use the theory of greedy lattice animals. Theorem~6.6 shows that without such an atom, for any $\epsilon>0$, the expected number of edges in $Geo(0,x)$ with weight within $\epsilon$ of the infimum of the support $I$ satisfies
\[
\mathbb{E}\#\{e \in Geo(0,x) : t_e \in [I,I+\epsilon]\} \leq C \|x\|_1 \beta(\epsilon)\ ,
\]
where $\beta(\epsilon) \to 0$ as $\epsilon \to 0$. Combining this with another dyadic partition of the interval $[I,\infty)$ (see Section~\ref{sec: finishing_deriv}) provides the required control on $(1-\log F(t_e))$ and allows the bound
\[
\mathbb{E} \sum_e (1-\log F(t_e)) \mathbf{1}_{\{e \in Geo(0,x)\}} \leq C\|x\|_1\ .
\]
Along with \eqref{eq: last_entropy_bound}, we obtain $\sum_{k=1}^\infty Ent(V_k^2) \leq C\|x\|$ and complete the proof.

\section{Entropy} \label{sec: entropy}
Recall the definition of entropy with respect to a probability measure $\mu$: 
\begin{df}
Let $(\Omega,\mathcal{F},\mu)$ be a probability space and $X\in L^1(\Omega,\mu)$ be nonnegative. Then
\[
Ent_\mu X = \mathbb{E}_\mu X\log X - \mathbb{E}_\mu X \log \mathbb{E}_\mu X .
\]
\end{df}
Note that by Jensen's inequality, $Ent_\mu X \geq 0$. We will make use of the variational characterization of entropy (see \cite[Section 5.2]{ledoux}):
\begin{prop}\label{prop: characterization}
If $X$ is nonnegative, then
\[
Ent_\mu (X) = \sup \{\mathbb{E}_\mu XY : \mathbb{E}_\mu e^Y \leq 1\}\ .
\]
\end{prop}

This characterization will let us prove the ``tensorization'' of entropy.
\begin{thm}\label{thm: tensorization}
Let $X$ be a non-negative $L^2$ random variable on a product probability space
\[\left(\prod_{i=1}^\infty \Omega_i,\mathcal{F}, \mu = \prod_{i=1}^\infty \mu_i \right),\]
where $\mathcal{F} = \bigvee_{i=1}^\infty \mathcal{G}_i$ and each triple $(\Omega_i,\mathcal{G}_i,\mu_i)$ is probability space. Then
\begin{equation}
Ent_\mu X \leq \sum_{k=1}^\infty \mathbb{E}_\mu Ent_{i}X\ ,\label{eqn: tensor}
\end{equation}
where $Ent_i X$ is the entropy of $X(\omega)=X(\omega_1,\ldots, \omega_i, \ldots ))$ with respect to $\mu_i$, as a function of the $i$-th coordinate (with all other values fixed).
\end{thm}
\begin{proof}
%Let $Z_n = \mathbb{E}_\mu[X \mid \mathcal{F}_n]$. Then applying tensorization on finite product spaces \cite[Section~4.13]{BLM} and Jensen's inequality, one has
%\[
%Ent_\mu Z_n \leq \sum_{i=1}^n \mathbb{E}_\mu Ent_i Z_n \leq \sum_{i=1}^n \mathbb{E}_\mu Ent_i X\ .
%\]
%Since $X \in L^2$, the sequence $(Z_n \log Z_n)$ is uniformly integrable and convergence of $Z_n \to X$ implies that $Ent_\mu~Z_n \to Ent_\mu X$. This completes the proof.
%We begin with the case when $X$ depends only on the first $n$ coordinates. In this case, the sum on the right of \eqref{eqn: tensor} is finite. 
We use a telescoping argument: write $\mathcal{F}_k$ for the sigma algebra generated by $\mathcal{G}_1\cup \cdots \cup \mathcal{G}_k$ (with $\mathcal{F}_0$ trivial) and compute for any $n$
\begin{align*}
Ent_\mu X &= \mathbb{E}_\mu X \left[ \log X - \log \mathbb{E}_\mu X \right] \\
&= \sum_{k=1}^n \mathbb{E}_\mu X \left[ \log \mathbb{E}_\mu [ X \mid \mathcal{F}_k] - \log \mathbb{E}_\mu [X \mid \mathcal{F}_{k-1}] \right] + \mathbb{E}_\mu X \left[ \log X - \log \mathbb{E}_\mu [X \mid \mathcal{F}_n] \right]\\
&= \sum_{k=1}^n \mathbb{E}_\mu \mathbb{E}_{\mu_k} X \left[ \log \mathbb{E}_\mu [ X \mid \mathcal{F}_k] - \log \mathbb{E}_\mu [X \mid \mathcal{F}_{k-1}] \right] \\
&+ \mathbb{E}_\mu X \left[ \log X - \log \mathbb{E}_\mu [ X \mid \mathcal{F}_n] \right]\ .
\end{align*}
Here $\mathbb{E}_{\mu_k}$ is expectation with respect to the coordinate $\omega_k$. Because for almost all realizations of $\{(\omega_i) : i \neq k\}$,
\[
\mathbb{E}_{\mu_k} \exp \left( \log \mathbb{E}_\mu [ X \mid \mathcal{F}_k] - \log \mathbb{E}_\mu [ X \mid \mathcal{F}_{k-1}] \right) = 1\ ,
\]
we use Proposition~\ref{prop: characterization} to get the bound
\[
Ent_\mu X \leq \sum_{k=1}^n \mathbb{E}_\mu Ent_k X + \mathbb{E}_\mu X \log X - \mathbb{E}_\mu X \log \mathbb{E}_\mu [ X \mid \mathcal{F}_n]\ .
\]
Putting $X_n = \mathbb{E}_\mu [ X \mid \mathcal{F}_n]$, one has
\[
\mathbb{E}_\mu X \log \mathbb{E}_\mu [X \mid \mathcal{F}_n] = \mathbb{E}_\mu X_n \log X_n\ .
\]
By martingale convergence (since $X \in L^1$), one has $X_n \to X$ almost surely. Furthermore, since $X \in L^2$, the sequence $(X_n \log X_n)$ is uniformly integrable. Therefore 
\[
\mathbb{E}_\mu X \log X - \mathbb{E}_\mu X \log \mathbb{E}_\mu [ X \mid \mathcal{F}_n] \to 0
\]
and the proof is complete.
\end{proof}

We end this section with the lower bound from Falik and Samorodnitsky \cite[Lemma~2.3]{FS}.
\begin{prop}[Falik-Samorodnitsky]\label{prop: lb}
If $X \geq 0$ almost surely,
\[
Ent_\mu (X^2) \geq \mathbb{E}_\mu X^2 \log \frac{\mathbb{E}_\mu X^2}{(\mathbb{E}_\mu X)^2}\ .
\]
\end{prop}
\begin{proof}
First assume $X > 0$ almost surely and define $Y = X/ \|X\|_2$. Then
\begin{align*}
Ent_\mu (Y^2) = \mathbb{E}_\mu Y^2 \log Y^2 - \mathbb{E}_\mu Y^2 \log \mathbb{E}_\mu Y^2 &= \mathbb{E}_\mu Y^2 \log Y^2 \\
& = -2 \mathbb{E}_\mu Y^2 \log (1/Y)\ .
\end{align*}
Apply Jensen to the measure $\mathbf{E}(\cdot)=\mathbb{E}_\mu (\cdot ~ Y^2)$ and the function $-\log$ to obtain
\[
Ent_\mu (Y^2) \geq -2\mathbb{E}_\mu Y^2 \log \frac{\mathbf{E}(1/Y)}{\mathbb{E}_\mu Y^2} = \mathbb{E}_\mu Y^2 \log \frac{\mathbb{E}_\mu Y^2}{(\mathbb{E}_\mu Y)^2}\ ,
\]
proving the proposition for $Y$. Now for $X$,
\[
Ent_\mu (X^2) = \|X\|_2^2 Ent_\mu (Y^2) \geq \|X\|_2^2 \mathbb{E}_\mu Y^2 \log \frac{\mathbb{E}_\mu Y^2}{(\mathbb{E}_\mu Y)^2} = \mathbb{E}_\mu X^2 \log \frac{\mathbb{E}_\mu X^2}{(\mathbb{E}_\mu X)^2}\ .
\]
If $X = 0$ with positive probability, we can conclude by a limiting argument applied to $X_n = \max\{1/n,X\}$.
\end{proof}

\section{Variance bound for $\tau(0,x)$} \label{sec: bernoulli}
The mechanism for sublinear behavior of the variance which was identified in \cite{BKS} can be understood as follows. Since a geodesic from the origin to $x$ is ``one-dimensional,'' one expects that most edges in the lattice have small probability to lie in it: the edges have small influence. This is not true of edges very close to the origin. To circumvent this difficulty, Benjamini, Kalai and Schramm considered an averaged version of the passage time (see \cite[Lemma~3]{BKS}), which they subsequently compare to the actual passage time from $0$ to $x$. It was brought to our attention by S. Sodin (see \cite[Section~3]{sodin13}) that their argument can be replaced by a geometric average. This observation was made earlier by K. Alexander and N. Zygouras in \cite{AZ} for polymer models. Let $x \in \mathbb{Z}^d$ and $B_m$ be a box of the form $[-m,m]^d$ for $m=\lceil \|x\|_1\rceil ^{1/4}$. Define
\begin{equation}\label{eqn: Fdef}
F_m = \frac{1}{\# B_m} \sum_{z \in B_m} \tau(z,z+x)\ .
\end{equation}
Note that by \eqref{eqn: 2log}, $\Var F_m < \infty$.

\subsection{Approximating $\tau(0,x)$ by $F_m$}
Because of the choice of $m$, the variance of $F_m$ closely approximates that of $\tau$:
\begin{prop}\label{prop: approximating}
Assume $\mathbb{E} t_e^2<\infty$. Then there exists $\mathbf{C}_3>0$ such that 
\[
|\Var \tau(0,x) - \Var F_m| \leq \mathbf{C}_3 \|x\|_1^{3/4} \text{ for all } x\ .
\]
\end{prop}
\begin{proof}
By subadditivity, for each $z \in B_m$, $|\tau(0,x) - \tau(z,z+x)| \leq \tau(0,z) + \tau(x,z+x)$. Therefore, writing $M_x = \max \{\tau(0,z) : z \in B_m \}$ and $\hat X = X - \mathbb{E} X$,
\[
|\Var \tau(0,x)-\Var F_m| \leq (\|\hat\tau(0,x)\|_2 + \|\hat F_m\|_2) \|\hat \tau(0,x) - \hat F_m\|_2\ .
\]
Using $\|\hat F_m\|_2 \leq \|\hat \tau(0,x)\|_2$, we get the bound
\[
2\|\hat \tau(0,x)\|_2 ( \|\tau(0,x) - F_m\|_2 + \mathbb{E} |\tau(0,x) - F_m|) \leq 4\|\hat \tau(0,x)\|_2 \|M_x\|_2\ .
\]
Since we assume \eqref{eqn: 2log}, \cite[Theorem~1]{kestenspeed} gives $\|\hat \tau(0,x)\|_2 \leq \mathbf{C}_4 \|x\|_1^{1/2}$. On the other hand, we can bound $M_x$ using the following lemma. 
\begin{lma}\label{lem: max_lemma}
If $\mathbb{E}t_e^2<\infty$, there exists $\mathbf{C}_5$ such that for all finite subsets $S$ of $\mathbb{Z}^d$,
\[
\mathbb{E} \left[\max_{x,y \in S} \tau(x,y)\right]^2 \leq \mathbf{C}_5 (\text{diam }S)^2\ .
\]
\end{lma}
\begin{proof}
We start with the argument of \cite[Lemma~3.5]{kesten}. Given $x,y \in S$, we can build $2d$ disjoint (deterministic) paths from $x$ to $y$ of length at most $\mathbf{C}_6 \|x-y\|_1$ for some integer $\mathbf{C}_6$. This means that $\tau(y,z)$ is bounded above by the minimum of $2d$ variables $T_1, \ldots, T_{2d}$, the collection being i.i.d. and each variable distributed as the sum of $\mathbf{C}_6 \text{diam}(S)$ i.i.d. variables $t_e$, so
\[
\mathbb{P}(\tau(x,y) \geq \lambda) \leq \prod_{i=1}^{2d} \mathbb{P}(T_i \geq \lambda) \leq \left[ \frac{\mathbf{C}_6\text{diam}(S)\Var t_e }{(\lambda-\mathbf{C}_6 \text{diam}(S) \mathbb{E} t_e)^2} \right]^{2d}\ .
\]
Therefore if we fix some $x_0 \in S$, for $M = \lceil 2\mathbf{C}_6 \mathbb{E} t_e \rceil$,
\[
\sum_{\lambda = M \text{diam}(S)}^\infty \lambda \max_{y \in S} \mathbb{P}(\tau(x_0,y) \geq \lambda) \leq (4 \mathbf{C}_6 \text{diam}(S) \Var t_e)^{2d} \sum_{\lambda=M \text{diam}(S)}^\infty \lambda^{1-4d} = \mathbf{C}_7 (\text{diam }S)^{2-2d}\ .
\]
%If $M>2\mathbf{C}_6\mathbb{E}t_e$, then the denominator of the integrand is at least $\lambda^2/4$, so we get an upper bound
%\[
%(4\mathbf{C}_6 Var(t_e) \text{diam}(S))^{2d} \int_{M\text{diam}(S)}^\infty \lambda^{1-4d} ~\text{d}\lambda = \mathbf{C}_7 (\text{diam }S)^{-2d}\ .
%\]
Last, by subadditivity,
\begin{align*}
\mathbb{E}\left[ \max_{x,y \in S} \tau(x,y) \right]^2 \leq 4 \mathbb{E} \left[ \max_{y \in S} \tau(x_0,y) \right]^2  &\leq 4(M \text{diam }S)^2 \\
&+ 8 (\text{diam }S)\sum_{\lambda=M\text{diam}(S)}^\infty \lambda \max_{y \in S} \mathbb{P}(\tau(x_0,y) \geq \lambda) \\
&\leq \mathbf{C}_8(\text{diam } S)^2\ .
\end{align*}
\end{proof}
Using the lemma, we find $\|M_x\|_2 \leq \mathbf{C}_9\text{diam}(B_m) \leq \mathbf{C}_{10}\|x\|_1^{1/4}$. This means
\[
|\Var \tau(0,x) - \Var F_m| \leq 4\mathbf{C}_4\mathbf{C}_{10} \|x\|_1^{1/2} \|x\|_1^{1/4} = \mathbf{C}_{11}\|x\|_1^{3/4}\ .
\]
\end{proof}

\subsection{Bounding the variance by the entropy}
Enumerate the edges of $\mathcal{E}^d$ as $e_1, e_2, \ldots$. We will bound the variance of $F_m$ using the martingale decomposition
\[F_m-\mathbb{E}F_m = \sum_{k=1}^\infty V_k,\]
where
\begin{equation} \label{eqn: Vkdef}
V_k = \mathbb{E}[F_m \mid \mathcal{F}_k] - \mathbb{E}[F_m \mid \mathcal{F}_{k-1}],
\end{equation}
and we have written $\mathcal{F}_k$ for the sigma-algebra generated by the weights $t_{e_1}, \ldots, t_{e_k}$ (with $\mathcal{F}_0$ trivial). In particular if $X\in L^1(\Omega, \mathbb{P})$, we have
\begin{equation}
\mathbb{E}[X\mid \mathcal{F}_k] = \int X\big((t_e)_{e\in \mathcal{E}^d}\big)\, \prod_{i\ge k+1}\mu(\mathrm{d}t_{e_i}).
\end{equation}

The idea now is to compare the variance of $F_m$ to $\sum_{k=1}^\infty Ent(V_k^2)$. The lower bound comes from the proof of \cite[Theorem~2.2]{FS}.

\begin{lma}[Falik-Samorodnitsky]\label{lem: lower_bound}
We have the lower bound
\begin{equation}\label{eq: BR_lower_bound}
\sum_{k=1}^\infty Ent(V_k^2) \geq \Var F_m ~ \log \left[ \frac{\Var F_m}{\sum_{k=1}^\infty (\mathbb{E} |V_k|)^2}\right] \ .
\end{equation}
\end{lma}
\begin{proof}
For $M \in \mathbb{N}$, define $\tilde F_m = \mathbb{E} [F_m \mid \mathcal{F}_M]$. We first use Proposition~\ref{prop: lb} and the fact that $\sum_{k=1}^M \mathbb{E} V_k^2 = \Var \tilde F_m$:
\[
\sum_{k=1}^M Ent(V_k^2) \geq \sum_{k=1}^M \mathbb{E} V_k^2 \log \left[ \frac{\mathbb{E}V_k^2}{(\mathbb{E} |V_k|)^2} \right] = - \Var \tilde F_m ~ \sum_{k=1}^M \frac{\mathbb{E}V_k^2}{\Var \tilde F_m} \log \left[ \frac{(\mathbb{E} |V_k|)^2}{\mathbb{E}V_k^2} \right]\ .
\]
Next use Jensen's inequality with the function $-\log$ and sum $\sum_{k=1}^M \frac{\mathbb{E} V_k^2}{\Var \tilde F_m} (\cdot)$ to get the lower bound
\[
-\Var \tilde F_m~ \log \left[ \sum_{k=1}^M \frac{\mathbb{E} V_k^2}{\Var \tilde F_m} \cdot \frac{(\mathbb{E} |V_k|)^2}{\mathbb{E} V_k^2} \right]\ ,
\]
which gives the lemma, after a limiting argument to pass to a countable sum.
\end{proof}

\section{Benaim and Rossignol's approach} \label{sec: BR}
In this section, we explain how the argument developed in \cite{benaimrossignol} can be extended, isolating a more general condition than the ``nearly gamma'' condition. It includes, for example, all power law distributions with $2+\epsilon$ moments. We emphasize that the content of this section is independent of the derivation of our main result.
In \cite{benaimrossignol}, the authors assume that the distribution
\[\mu = h(x)\,\mathrm{d}x,~ h \text{ continuous}\]
is an absolutely continuous measure such that 
\[(\operatorname{supp} h)^\circ := \{x:h(x)>0\} \subset (0,\infty)\]
is an interval. Denoting by $G(x)$ the distribution function of the standard normal distribution, $H(x) =\int_{-\infty}^xh(t)\,\mathrm{d}t$, and $X$ an $N(0,1)$ variable, the random variable
\begin{equation}\label{eqn: cov}
Y=T(X),
\end{equation}
with $T=H^{-1}\circ G$, has distribution $\mu$. Recall the Gaussian logarithmic Sobolev inequality \cite{federbusch, gross, stam}: for any smooth $f:\mathbb{R}\rightarrow \mathbb{R}$
\begin{equation}
\label{eqn: gaussianLSI}
\mathbb{E} f^2(X)\log \frac{f^2(X)}{\mathbb{E}f^2(X)} \le 2\mathbb{E}(f'(X))^2.
\end{equation}
Combining \eqref{eqn: cov} and \eqref{eqn: gaussianLSI}, a calculation yields
\begin{equation}\label{eq: generalLSI}
Ent_\mu (f(Y))^2 \le 2\mathbb{E}_\mu((\psi\cdot f')(Y))^2,
\end{equation}
where
\[\psi(Y) = \frac{(g\circ G^{-1} \circ H)(Y)}{h(Y)}\]
for any $f$ in a suitable Sobolev space.

Benaim and Rossignol apply this inequality to the passage time, using inequality \eqref{eq: BR_lower_bound}. 
It is shown in \cite{benaimrossignol}, along the same lines as the proof of Lemma \ref{eqn: entropybydelta} that \eqref{eq: generalLSI} implies
\begin{equation}\label{eqn: applyLSI}
\sum_{k=1}^\infty Ent_\mu(V_k^2)\le 2\sum_{j=1}^\infty \mathbb{E}[(\psi(t_{e_j})\partial_{t_{e_j}}F_m)^2],
\end{equation}
with $F_m$ as in \eqref{eqn: Fdef}. The derivative with respect to the edge weight can be expressed as
\begin{equation}\label{eqn: Fderivative}
\partial_{t_{e_j}}F_m = \frac{1}{\sharp B_m}\sum_{z\in B_m}\mathbf{1}_{\{e_j\in Geo(z,z+x)\}}\ .
\end{equation}
Observe that the right side of \eqref{eqn: Fderivative} is a decreasing function of the edge weight $t_{e_j}$.  

The following simple asymptotics appear in \cite[Lemma~5.2]{benaimrossignol}:
\begin{lma}
\begin{align}
\label{eqn: Gasymp}
g\circ G^{-1}(y) &\sim y \sqrt{-2\log y}, \quad y\rightarrow 0,\\
g\circ G^{-1}(y) &\sim (1-y)\sqrt{-2\log(1-y)}, \quad y\rightarrow 1.
\end{align}
That is, in each case the ratio of the left to the right side tends to 1.
\end{lma}

Suppose that there is a constant $\mathbf{C}_{12}>0$ such that 
\begin{equation}
\frac{H(t)\sqrt{-\log t}}{h(t)}\le \mathbf{C}_{12} \label{eqn: leftendpoint}
\end{equation}
for all $t$ with $I\leq t \le I+ \delta$, with $\delta>0$ and $I$ the left endpoint of the interval $(\operatorname{supp} h)^\circ$ (as in \eqref{eq: I_def}). The condition \eqref{eqn: leftendpoint} holds, for example, if the density $h$ is monotone near $I$ or if $h(t)\asymp (t-I)^\alpha$ for some (integrable) power $\alpha$. The latter condition appears in \cite[Lemma~5.3]{benaimrossignol}. 

For $M>0$ such that $F(M)<1$, the expectation in \eqref{eqn: applyLSI} can be computed as
\begin{align}
\mathbb{E}[(\psi(t_{e_j})\partial_{t_{e_j}}F_m)^2] &= \mathbb{E}\mathbb{E}_{\mu_j}[(\psi(t_{e_j})\partial_{t_{e_j}}F_m)^2] \nonumber \\
&=  \mathbb{E}\mathbb{E}_{\mu_j}[(\psi(t_{e_j})\partial_{t_{e_j}}F_m)^2; t_{e_j} \le M]+ \mathbb{E}\mathbb{E}_{\mu_j}[(\psi(t_{e_j})\partial_{t_{e_j}}F_m)^2; t_{e_j} > M]\ .\label{eqn: split}
\end{align}
For $t_{e_j}\le M$, \eqref{eqn: Gasymp} implies that the first term in \eqref{eqn: split} is bounded by
\[\left(\max\left\{ \mathbf{C}_{12},\sup_{\delta \le t\le M} h(t)^{-1}\right\}\right)^2\cdot \mathbb{E}_{\mu_j}(\partial_{t_{e_j}}F_m)^2.\]
The maximum is finite by assumption, and we have, by Cauchy-Schwarz,
\[\mathbb{E}(\partial_{t_{e_j}}F_m)^2 \le \frac{1}{\sharp B_m}\sum_{z\in B_m}\mathbb{E}(\mathbf{1}_{\{e_j\in Geo(z,z+x)}\}).\]
From there, one can conclude the argument as in Sections~\ref{sec: finishing_deriv} and \ref{sec: proof}.

As for the second term in \eqref{eqn: split}, assume first that
\begin{equation}\label{eqn: sqrtnearlygamma}
\psi(t_{e_j})\le \mathbf{C}_{13}\sqrt{t_{e_j}}.
\end{equation}
This is the ``nearly gamma'' condition of Benaim and Rossignol. The right side of \eqref{eqn: sqrtnearlygamma} is increasing in $t_{e_j}$. Using this in \eqref{eqn: split} together with the Chebyshev association inequality \cite[Theorem~2.14]{BLM}, we find
\begin{align} \label{eqn: fkgapplication}
 \mathbb{E}\mathbb{E}_{\mu_j}[(\psi(t_{e_j})\partial_{t_{e_j}}F_m)^2; t_{e_j} > M] &\le  \mathbf{C}^2_{13}\mathbb{E}\mathbb{E}_{\mu_j}(\sqrt{t_{e_j}}\cdot \partial_{t_{e_j}}F_m)^2\\
&\le \mathbf{C}_{13}^2\mathbb{E}(t_{e_j})\cdot\mathbb{E}(\partial_{t_{e_j}}F_m)^2. \nonumber
\end{align}
The previous argument shows that the condition \eqref{eqn: sqrtnearlygamma} is not necessary: it is sufficient that $\psi$ be bounded by some increasing, square integrable function of $t_{e_j}$. Suppose for example that $t\mapsto h(t)$ is decreasing for $t>M$. In this case, by \eqref{eqn: Gasymp}, we have
\begin{align}
\psi(t_{e_j})\mathbf{1}_{\{t_{e_j} > M\}} &= \frac{(g\circ G^{-1} \circ H)(t_{e_j})}{h(t_{e_j})}\mathbf{1}_{\{t_{e_j} > M\}} \nonumber\\
&\le \mathbf{C}_{14} \frac{(1-H(t_{e_j}))\cdot \sqrt{-2\log(1-H(t_{e_j})})}{h(t_{e_j})}\mathbf{1}_{\{t_{e_j} > M\}}. \label{eqn: big}
\end{align}

Let us denote by $K(t_{e_j})$ the expression in \eqref{eqn: big}. For $t>M$, we have
\begin{align}\label{eqn: weights}
1-H(t) = \int_t^\infty h(s)\,\mathrm{d}s &= \int_t^\infty s^{2/3+\epsilon} s^{-2/3-\epsilon} h(s)\,\mathrm{d}s \nonumber \\
&\le \left( \int_t^\infty s^{2+3 \epsilon} \, h(s)\mathrm{d}s\right)^{1/3}\left(\int_t^\infty s^{-1-3\epsilon/2}h(s)\,\mathrm{d}s\right)^{2/3} \nonumber \\
&\le \mathbf{C}_{15} h(t)^{2/3},\nonumber
\end{align}
assuming $h(s)$ is decreasing for $s> M$ and that the distribution posesses $2+3\epsilon$ moments. We have used the $L^3-L^{3/2}$ H\"older inequality. This gives
\[K(t) \le \mathbf{C}_{14}\mathbf{C}_{15} h(t)^{-1/3}\sqrt{-2\log(1-H(t))} \cdot \mathbf{1}_{\{t>M\}}.\]
Thus $K(t)$ is bounded by a quantity which is increasing in $t$. Using the Chebyshev association inequality as in \eqref{eqn: fkgapplication}, we find
\[\mathbb{E}\mathbb{E}_{\mu_j}[(\psi(t_{e_j})\partial_{t_{e_j}}F_m)^2; t_{e_j} > M] \le  (\mathbf{C}_{14}\mathbf{C}_{15})^2\mathbb{E}\left(h^{-1/3}(t_{e_j}) \sqrt{-2\log(1-H(t_{e_j})}) \right)^2\cdot\mathbb{E}(\partial_{t_{e_j}}F_m)^2.\]
We are left with the task of estimating the first expectation, which is
\[\int h(s)^{-2/3}(-2\log(1-H(s)) h(s)\,\mathrm{d}s=\int h(s)^{1/3}(-2\log(1-H(s))\,\mathrm{d}s.\]
We again use polynomial weights and $L^3-L^{3/2}$:
\begin{align*}
\int h(s)^{1/3}(-2\log(1-H(s)))\,\mathrm{d}s &= \int s^{-2/3-\epsilon} s^{2/3+\epsilon}h(s)^{1/3}(-2\log(1-H(s)))\,\mathrm{d}s\\
&\le  \left(\int s^{-1-3\epsilon/2}\,\mathrm{d}s\right)^{2/3}\left(\int s^{2+3\epsilon}(-2\log(1-H(s)))^3h(s)\,\mathrm{d}s\right)^{1/3}.
\end{align*}
A further application of H\"older's inequality allows one to control the logarithm, at the cost of an arbitrarily small increase in the moment assumption. It follows that
\[\mathbb{E}(\psi(t_{e_j})\partial_{t_{e_j}}F_m)^2\le \mathbf{C}_{16}\mathbb{E}(\partial_{t_{e_j}}F_m)^2 \]
if the distribution $\mu$ has $2+\epsilon'$ moments. In conclusion, Benaim and Rossignol's argument extends to the case of distributions with $2+\epsilon$ moments whose densities are positive and eventually decreasing.

One can derive many variants of the above, the key point being the application of positive association in \eqref{eqn: fkgapplication}.

\section{The lower bound} \label{sec: lower}
In this section we derive the first generalization of Kesten's geodesic length estimate and show how it is used to bound the sum $\sum_{k=1}^\infty (\mathbb{E}|V_k|)^2$ appearing in \eqref{eq: BR_lower_bound}. Let $\mathcal{G}$ be the set of all finite self-avoiding geodesics.

\begin{lma}\label{lem: exp_intersections}
Assuming \eqref{eqn: 2log} and \eqref{eqn: geodesics}, there exists $\mathbf{C}_{17}>0$ such that for all $x$ and all finite $E \subset \mathcal{E}^d$,
\[
\mathbb{E} \max_{\gamma \in \mathcal{G}} \# (E \cap \gamma) \leq \mathbf{C}_{17} \text{diam}(E)\ .
\]
\end{lma}
\begin{proof}
%For any path $\gamma$ and $a>0$ let $N_a(\gamma)$ be the number of edges in $\gamma$ with weight at least $a$. 
Choose $a, \mathbf{C}_2>0$ from Theorem~\ref{thm: kesten_exp}.
%to get $a,\mathbf{C}_2>0$ such that
%\[
%\mathbb{P}(\exists \text{ self-avoiding }\gamma \text{ starting at } 0 \text{ with } \# \gamma \geq n \text{ but } N_a(\gamma) < a \# \gamma) \leq \exp(-\mathbf{C}_2 n)\ .
%\]
If $\# (E \cap \gamma) \geq \lambda$ for some $\gamma \in \mathcal{G}$, then we may find the first and last intersections (say $y$ and $z$ respectively) of $\gamma$ with $V$, the set of endpoints of edges in $E$. The portion of $\gamma$ from $y$ to $z$ is then a geodesic with at least $\lambda$ edges. 
%If, in addition, every path starting at $y$ of length at least $\lambda$ has at least $a \lambda$ edges with weight at least $a$, then $\tau(y,z) \geq a^2 \lambda$. 
This means
\[
\mathbb{P}(\# (E \cap \gamma) \geq \lambda \text{ for some }\gamma \in \mathcal{G}) \leq (\# V) \exp(-\mathbf{C}_2 \lambda) + \mathbb{P}\left( \max_{y,z \in V}\tau(y,z) \geq a \lambda \right)\ .
\]
Therefore
\[
\mathbb{E} \max_{\gamma \in \mathcal{G}} \#(E \cap \gamma) \leq \text{diam}(E) + \sum_{\lambda=\text{diam}(E)}^\infty (\#V) \exp(-\mathbf{C}_2 \lambda) + \sum_{\lambda =  \text{diam}(E) }^\infty \mathbb{P}\left( \max_{y,z \in V}\tau(y,z) \geq a \lambda \right)\ .
\]
By the inequality $\text{diam}(E) \geq \mathbf{C}_{18}(\#V)^{1/d}$ for some universal $\mathbf{C}_{18}$, the middle term is bounded uniformly in $E$, so we get the upper bound
\[
\mathbf{C}_{19}\text{diam}(E) + \frac{1}{a} \mathbb{E} \max_{y,z \in V} \tau(y,z)\ .
\]
By Lemma~\ref{lem: max_lemma}, this is bounded by $\mathbf{C}_{20} \text{diam}(E)$.
\end{proof}

We will now apply Lemma~\ref{lem: exp_intersections} to get an upper bound on $\sum_{k=1}^\infty (\mathbb{E}|V_k|)^2$. To do so, we use a simple lemma, a variant of which is already found in various places, including the work of Benaim-Rossignol \cite[Lemma~5.9]{benaimrossignol}. For its statement, we write an arbitrary element $\omega \in \Omega$ as $(t_{e^c},t_e)$, where $t_{e^c} = (t_f : f \neq e)$. Further, set
\[
S := \sup \text{supp}(\mu) = \sup \{x : F(x)<1\} \in \mathbb{R} \cup \{\infty\}\ .
\]
We use the following short-hand:
\[
\tau_z = \tau(z,z+x)\ .
\]
\begin{lma}\label{lem: deriv}
For $e \in \mathcal{E}^d$ and $z \in \mathbb{Z}^d$, the random variable
\[
D_{z,e} := \sup \{r < S : e \text{ is in a geodesic from }z \text{ to } z+x \text{ in } (t_{e^c},r)\}\
\] 
has the following properties almost surely.
\begin{enumerate}
\item $D_{z,e} < \infty$.
\item For $s \leq t < S$,
\[
\tau_z(t_{e^c}, t) - \tau_z(t_{e^c},s) = \min\{t-s, (D_{z,e}-s)_+\}\ .
\]
\item For $s<D_{z,e}$, $e\in Geo(z,z+x)$ in $(t_{e^c},s)$.
\end{enumerate}
\end{lma}
\begin{proof}
Part 1 is clear if $S < \infty$. Otherwise choose any path $\gamma$ not including $e$. Then for $r$ larger than the passage time of this path, $e$ cannot be in a geodesic in $(t_{e^c},r)$, giving $D_{z,e} < \infty$.

If $e$ is in a geodesic $\gamma$ in $(t_{e^c},t)$ and $t \geq s$ then the passage time of $\gamma$ decreases by $t-s$ in $(t_{e^c},s)$. Since the passage time of no other path decreases by more than $t-s$, $\gamma$ is still a geodesic in $(t_{e^c},s)$. This shows that 
\begin{equation}\label{eq: dec}
\mathbf{1}_{\{e \text{ is in a geodesic from } z \text{ to } z+x\}} \text{ is a non-increasing function of }t_e\ .
\end{equation}
Therefore if $t < D_{z,e}$, $e$ is in a geodesic in $(t_{e^c},t)$ and by the above argument, for any $s \leq t$, part 2 holds. We can then extend to $s \leq t \leq D_{z,e}$ by continuity.

If $D_{z,e} < s \leq t$ then $e$ is not in a geodesic from $z$ to $z+x$ in $(t_{e^c},s)$. By \eqref{eq: exist_geodesics}, we can almost surely find a geodesic $\gamma$ in $(t_{e^c},s)$ not containing $e$ and this path has the same passage time in $(t_{e^c},t)$. However all other paths have no smaller passage time, so $\tau_z(t_{e^c},t) - \tau_z(t_{e^c},s) = (D_{z,e}-s)_+$ almost surely, proving part 2 in this case. We can then extend the result to $D_{z,e} \leq s \leq t$ by continuity and for $s \leq D_{z,e} \leq t$ write
\[
\tau_z(t_{e^c},t) - \tau_z(t_{e^c},s) = \tau_z(t_{e^c},t) - \tau_z(t_{e^c},D_{z,e}) + \tau_z(t_{e^c},D_{z,e}) - \tau_z(t_{e^c},s)\ , 
\]
and use the other cases to complete the proof.

For part 3, let $s<D_{z,e}$, so that by \eqref{eq: dec}, $e$ is in a geodesic $\gamma_1$ in $(t_{e^c},\frac{s+D_{z,e}}{2})$ from $z$ to $x+z$. Assume for a contradiction that $e$ is not in every geodesic from $z$ to $x+z$ in $(t_{e^c},s)$, and choose $\gamma_2$ as one that does not contain $e$. Because $\frac{s + D_{z,e}}{2} \geq s$, $\gamma_1$ is still a geodesic in $(t_{e^c},s)$ and therefore has the same passage time in this configuration as $\gamma_2$. But then in $(t_{e^c},\frac{s+D_{z,e}}{2})$ it has strictly larger passage time, contradicting the fact that it is a geodesic.
\end{proof}

\begin{prop}\label{prop: 5.3}
Assuming \eqref{eqn: 2log} and \eqref{eqn: geodesics}, there exists $\mathbf{C}_{21}$ such that
\[
\sum_{k=1}^\infty (\mathbb{E}|V_k|)^2 \leq \mathbf{C}_{21} \|x\|_1^{\frac{5-d}{4}} \text{ for all } x\ .
\]
\end{prop}

\begin{proof}
Using the definition of $V_k$,
\begin{align}
\mathbb{E}|V_k| &= \frac{1}{\# B_m} \mathbb{E} \left| \mathbb{E} \left[ \sum_{z \in B_ m} \tau_z \mid \mathcal{F}_k \right] - \mathbb{E} \left[ \sum_{z \in B_m} \tau_z \mid \mathcal{F}_{k-1} \right] \right| \nonumber \\
&\leq \frac{1}{\# B_m} \sum_{z \in B_m} \mathbb{E} \left| \mathbb{E} \left[ \tau_z \mid \mathcal{F}_k \right] - \mathbb{E} \left[ \tau_z \mid \mathcal{F}_{k-1} \right] \right| \label{eq: tomato_salsa}\ .
\end{align}

Write a configuration $\omega$ as $(t_{<k},t_{e_k},t_{>k})$, where 
\[
t_{<k} = (t_{e_j}:j < k) \text{ and } t_{>k} = (t_{e_j} : j > k)\ .
\]
The summand in \eqref{eq: tomato_salsa} becomes
\begin{align*}
&~\int \left| \int \tau_z(t_{<k},t,t_{>k}) \mathbb{P}(\text{d}t_{>k}) - \int \tau_z(t_{<k},s,t_{>k}) \mu(\text{d}s)\mathbb{P}(\text{d}t_{>k}) \right| \mu(\text{d}t) \mathbb{P}(\text{d}t_{<k}) \\
\leq 2 &~ \mathbb{E} \iint_{t \geq s} \left| \tau_z(t_{<k},t,t_{>k}) - \tau_z(t_{<k},s,t_{>k}) \right| \mu(\text{d}s) \mu(\text{d}t)\ .
\end{align*}
By Lemma~\ref{lem: deriv} and $\mathbb{E}t_e<\infty$, this equals
\[
\mathbb{E} \iint_{t \geq s} \min\{ t-s, (D_{z,e_k}-s)_+\} \mu(\text{d}s) \mu(\text{d}t) \leq 2 \int t~ \int_{s < D_{z,e_k}} \mu(\text{d}s) \mu(\text{d}t) = \mathbf{C}_{22} F(D_{z,e_k}^-)\ .
\]
Using part 3 of the same lemma, $\mathbb{E} F(D_{z,e_k}^-) \leq \mathbb{P}(e_k \in Geo(z,z+x))$. Therefore
\[
\mathbb{E}|V_k| \leq \mathbf{C}_{22} \frac{1}{\#B_m} \sum_{z \in B_m} \mathbb{P}(e_k \in Geo(z,z+x))\ .
\]
By translation invariance, the above probability equals $\mathbb{P}(e_k+z \in Geo(0,x))$, so we get the bound $\frac{\mathbf{C}_{22}}{\#B_m} \mathbb{E} \# \left[ Geo(0,x) \cap \{e_k+z : z \in B_m\} \right]$. Lemma~\ref{lem: exp_intersections} provides $\mathbf{C}_{23}$ such that this is no bigger than $\mathbf{C}_{23}\frac{\text{ diam }B_m}{ \#B_m}$. Hence
\[
\mathbb{E} |V_k| \leq \mathbf{C}_{24} \|x\|_1^{\frac{1-d}{4}}\ .
\] 
This leads to
\[
\sum_{k=1}^\infty (\mathbb{E} |V_k|)^2 \leq \mathbf{C}_{24} \|x\|_1^{\frac{1-d}{4}} \sum_{k=1}^\infty \mathbb{E}|V_k| \leq \mathbf{C}_{25} \|x\|_1^{\frac{1-d}{4}} \frac{1}{\# B_m}\sum_{z\in B_m} \sum_{k=1}^\infty \mathbb{P}(e_k \in Geo(z,z+x)) \leq \mathbf{C}_{26} \|x\|_1^{\frac{5-d}{4}}\ .
\]
In the last inequality we have used Theorem~\ref{thm: kesten_length}.

\end{proof}

\section{Sublinear variance for general distributions}\label{sec: derivatives}

Combining the results from the previous sections, we have shown so far that if \eqref{eqn: 2log} and \eqref{eqn: geodesics} hold then
\begin{equation}\label{eq: breakdown}
\Var  \tau(0,x) \leq \Var F_m + \mathbf{C}_3\|x\|_1^{3/4} \leq \mathbf{C}_{3}\|x\|_1^{3/4} + \left[ \log \left[ \frac{\Var F_m}{\|x\|_1^{\frac{5-d}{4}}} \right] \right]^{-1} \sum_{k=1}^\infty Ent(V_k^2)\ .
\end{equation}
Our goal now is to bound the sum by $C\|x\|_1$. We will do this using a Bernoulli encoding.  

\subsection{Bernoulli encoding} \label{sec: encoding}
We will now view our edge variables as the push-forward of a Bernoulli sequence. Specifically, for each edge $e$, let $\Omega_e$ be a copy of $\{0,1\}^\mathbb{N}$ with the product sigma-algebra. We will construct a measurable map $T_e:\Omega_e \to \mathbb{R}$ using the distribution function $F$. To do this, we create a sequence of partitions of the support of $\mu$. Recalling $I := \inf \text{supp}(\mu) = \inf \{x : F(x) > 0\}$, set
\[
a_{0,j} = I \text{ and }a_{i,j} = \min\left\{x : F(x) \geq \frac{i}{2^j}\right\} \text{ for } j \geq 1 \text{ and } 1 \leq i \leq 2^j-1\ .
\]
Note that by right continuity of $F$, the minimum above is attained; that is,
\begin{equation}\label{eq: rt_cty}
F(a_{i,j}) \geq \frac{i}{2^j} \text{ for } j \geq 1 \text{ and } 0 \leq i \leq 2^j -1\ .
\end{equation}
Let us note two properties of the sequence.
\begin{equation}\label{eq: partition}
\text{For }j \geq 1,~ a_{0,j} \leq a_{1,j} \leq \cdots \leq a_{2^j-1,j}\ .
\end{equation}
\begin{equation}\label{eq: if_f}
\text{For }i=0, \ldots, 2^j-1, ~ x \geq a_{i,j} \text{ if and only if }F(x) \geq \frac{i}{2^j} \text{ and } x \geq a_{0,j}\ .
\end{equation} 
%\begin{proof}
%The first display is clear, so we prove the second. Suppose that $x\geq a_{i,j}$ for some $i=0, \ldots, 2^j-1$. Then $x \geq a_{0,j}$ because $(a_{i,j})$ is non-decreasing in $i$. By \eqref{eq: rt_cty} and monotonicity of $F$, $F(x) \geq i/2^j$. Conversely assume that $F(x) \geq i/2^j$ and $x \geq a_{0,j}$ for some $i=0, \ldots, 2^j-1$. In the case $i=0$, we have $x \geq a_{0,j}$. Otherwise if $i>0$ then $x \in \{y : F(y) \geq i/2^j\}$, so $x$ is larger than the minimum, $a_{i,j}$.
%\end{proof}

%Next define the intervals
%\[
%I_{i,j} = [a_{i,j},a_{i+1,j}] \text{ for } j \geq 1 \text{ and } 0 \leq i \leq 2^j-1\ .
%\]
Each $\omega \in \Omega_e$ gives us an ``address'' for a point in the support of $\mu$. Given $\omega = (\omega_1, \omega_2, \ldots)$ and $j \geq 1$, we associate a number $T_j(\omega)$ by
\[
T_j(\omega) = a_{i(\omega,j),j}, \text{ where } i(\omega,j) = \sum_{l=1}^j 2^{j-l} \omega_l\ .
\]
$i(\omega,j)$ is just the number between $0$ and $2^j - 1$ that corresponds to the binary number $\omega_1 \cdots \omega_j$. It will be important to note that if $\omega_i \leq \hat \omega_i$ for all $i \geq 1$ (written $\omega \leq \hat \omega$), then $i(\omega,j) \leq i(\hat \omega,j)$ for all $j \geq 1$. This, combined with the monotonicity statement \eqref{eq: partition}, implies
\begin{equation}\label{eq: monotone_2}
\omega \leq \hat \omega \Rightarrow T_j(\omega) \leq T_j(\hat \omega) \text{ for all } j \geq 1\ .
\end{equation}

It is well-known that one can represent Lebesgue measure on $[0,1]$ using binary expansions and Bernoulli sequences. One way to view the encoding $T$ in Lemma~\ref{lem: encoding} is a composition of this representation with the right-continuous inverse of the distribution function $F$. The function $T_j$ instead uses an inverse approximated by simple functions taking dyadic values.

\begin{lma}\label{lem: encoding}
For each $\omega$, the numbers $(T_j(\omega))$ form a non-decreasing sequence and have a limit $T(\omega)$. This map $T:\Omega_e \to \mathbb{R} \cup \{\infty\}$ is measurable and has the following properties.
\begin{enumerate}
\item (Monotonicity) If $\omega \leq \hat \omega$ then $T(\omega) \leq T(\hat \omega)$.
\item (Nesting) For any $\omega \in \Omega_e$ and $j \geq 1$, if $i(\omega,j) < 2^j-1$ then
\[
a_{i(\omega,j),j} \leq T(\omega) \leq a_{i(\omega,j)+1,j}\ .
\]
\item If $\omega_k = 0$ for some $k \geq 1$ then $T(\omega) < \infty$.
\item Letting $\pi$ be the product measure $\prod_{l \in \mathbb{N}} \pi_l$, with each $\pi_l$ uniform on $\{0,1\}$, we have
\[
\pi \circ T^{-1} = \mu\ .
\]
By part 3, $T$ is $\pi$-almost surely finite.
\end{enumerate}
\end{lma}
\begin{proof}
The functions $T_j$ are each measurable since their ranges are finite and the pre-image of each point is a cylinder in $\Omega_e$. If we show that $T_j \to T$ pointwise then $T$ will also be measurable. Given $\omega \in \Omega_e$, we have 
\[
\frac{i(\omega,j)}{2^j} = \frac{1}{2^j} \sum_{l=1}^j 2^{j-l}\omega_l = \sum_{l=1}^j 2^{-l} \omega_l \leq \sum_{l=1}^{j+1} 2^{-l} \omega_l = \frac{i(\omega,j+1)}{2^{j+1}}\ .
\]
Therefore if $x$ is such that $F(x) \geq \frac{i(\omega,j+1)}{2^{j+1}}$ then also $F(x) \geq \frac{i(\omega,j)}{2^j}$. This means that if $i(\omega,j) > 0$,
\[
T_j(\omega) = \min \left\{x : F(x) \geq \frac{i(\omega,j)}{2^j} \right\} \leq \min \left\{x : F(x) \geq \frac{i(\omega,j+1)}{2^{j+1}} \right\}= T_{j+1}(\omega)\ .
\]
Otherwise if $i(\omega,j)=0$ then $T_{j+1}(\omega) \geq a_{0,j+1}=a_{0,j}=T_j(\omega)$. In either case, $(T_j(\omega))$ is monotone and has a limit $T(\omega)$.

For part 1, we simply take limits in \eqref{eq: monotone_2}. To prove part 2, we note the lower bound follows from monotonicity. For the upper bound, take $\omega \in \Omega_e$ and let $k \geq j$. Then
\[
\frac{i(\omega,k)}{2^k} = \sum_{l=1}^k 2^{-l} \omega_l \leq \sum_{l=1}^j 2^{-l} \omega_l + \sum_{l=j+1}^\infty 2^{-l} = \frac{i(\omega,j)+1}{2^j} \leq \frac{2^j-1}{2^j}\ .
\]
If $\omega$ is the zero sequence then $T(\omega) = I$ and $T(\omega) \leq a_{i(\omega,j)+1,j}$. Otherwise we can find $k \geq j$ such that $i(\omega,k) \neq 0$. For this $k$, $F(x) \geq \frac{i(\omega,j)+1}{2^j}$ implies $F(x) \geq \frac{i(\omega,k)}{2^k}$, giving 
\[
T_k(\omega) = a_{i(\omega,k),k} \leq a_{i(\omega,j)+1,j}\ .
\]
Taking the limit in $k$ gives the result.

In part 3, we assume that $\omega_k=0$ for some $k \geq 1$. Then $i(\omega,k+1) < 2^{k+1}-1$ and therefore by part 2, 
\[
T(\omega) \leq a_{i(\omega,k+1)+1,j} \leq a_{2^{k+1}-1,j} <\infty\ .
\]

Last we must show that $\pi \circ T^{-1} = \mu$. The first step is to show that for each $x \in \mathbb{R}$,
\[
\pi \circ T_j^{-1} ((-\infty,x]) \to \pi \circ T^{-1}((-\infty,x])\ .
\]
Consider the sets
\[
S_j(x) = \{\omega \in \Omega_e : T_j(\omega) \leq x\}\ .
\]
If $T_{j+1}(\omega) \leq x$ then $T_j(\omega) \leq T_{j+1}(\omega) \leq x$, so these sets are decreasing. If $\omega$ is in their intersection then $T_j(\omega) \leq x$ for all $j$. Since $T_j(\omega) \to T(\omega)$ this means $T(\omega) \leq x$ and thus $\omega \in S(x) := \{\omega \in \Omega_e : T(\omega) \leq x\}$. Conversely, if $\omega \in S(x)$ then $T(\omega) \leq x$ and so $T_j(\omega) \leq T(\omega) \leq x$ for all $j$, meaning $\omega \in \cap_j S_j(x)$. Therefore $\pi \circ T_j^{-1} ((-\infty,x])$ converges to $\pi \circ T^{-1}((-\infty,x])$.

Next we claim that
\begin{equation}\label{eq: equiv}
x \geq a_{0,j} \Rightarrow \pi \circ T_j^{-1} ((-\infty,x]) = 2^{-j} \max \left\{i+1 : F(x) \geq \frac{i}{2^j} \right\}\ .
\end{equation}
The left side of the equality is $\pi(\{\omega : T_j(\omega) \leq x\})$. The function $T_j$ is constant on sets of $\omega$ with the same first $j$ entries. By definition, if $\omega$ has first $j$ entries $\omega_1 \cdots \omega_j$ then $T_j(\omega) = T_j(\omega_1 \cdots \omega_j) = a_{i(\omega,j),j}$. So
\[
\pi\circ T_j^{-1}((-\infty,x]) = 2^{-j} \# \left\{ (\omega_1, \cdots, \omega_j) : a_{i(\omega,j),j} \leq x\right\}\ .
\]
Also, since $x \geq a_{0,j}$, \eqref{eq: if_f} gives
\[
\pi\circ T_j^{-1} ((-\infty,x]) = 2^{-j} \# \left\{(\omega_1, \cdots, \omega_j) : F(x) \geq \sum_{l=1}^j 2^{-l} \omega_l \right\}\ .
\]
This is exactly the right side of \eqref{eq: equiv}.

By \eqref{eq: equiv}, $\left| \pi \circ T_j^{-1}((-\infty,x]) - F(x) \right| \leq 2^{-j}$ and so $\pi\circ T_j^{-1} ((-\infty,x]) \to F(x)$, completing the proof of part 4.

\end{proof}

\subsection{Bound on discrete derivatives}
In this section we prove the result:
\begin{thm}\label{thm: derivative_bound}
Assume \eqref{eqn: 2log} and \eqref{eqn: geodesics}. There exists $\mathbf{C}_{27}$ such that
\[
\sum_{k=1}^\infty Ent(V_k^2) \leq \mathbf{C}_{27}\|x\|_1\ .
\]
\end{thm}

The proof will be broken into subsections. In the first we apply Bonami's inequality to the Bernoulli encoding of $F_m$ to get a sum involving discrete derivatives. The next subsection uses the quantities $D_{z,e}$ from Lemma~\ref{lem: deriv} to control the sum of derivatives. In the third subsection, we give a lemma based on the theory of greedy lattice animals and in the final subsection, we use this lemma to achieve the bound $\mathbf{C}_{27}\|x\|_1$.

\subsubsection{Application of Bonami's inequality}

We will view $F_m$ as a function of sequences of Bernoulli variables, so define
\[
\Omega_B = \prod_e \Omega_e
\]
where $\Omega_e$ is, as in the last section, a copy of $\{0,1\}^\mathbb{N}$. The measure on $\Omega_e$ is $\pi_e$, a product of the form $\prod_{j \geq 1} \pi_{e,j}$ with $\pi_{e,j}$ uniform on $\{0,1\}$ and the measure on $\Omega_B$ is $\pi : = \prod_e \pi_e$. Here as usual we use the product sigma-algebra. A typical element of $\Omega_B$ is denoted $\omega_B$ and we list the collection of individual Bernoulli variables as 
\[
\omega_B = \left\{ \omega_{e,j} : e \in \mathcal{E}^d,~ j \geq 1\right\}\ .
\]
Last, calling $T_e$ the map from Lemma~\ref{lem: encoding} on $\Omega_e$, the product map $T:=\prod_e T_e : \Omega_B \to \Omega$ is defined
\[
T (\omega_B) = (T_e(\omega_e) : e \in \mathcal{E}^d)\ .
\]
It is measurable and, by Lemma~\ref{lem: encoding}, pushes the measure $\pi$ forward to $\mathbb{P}$, our original product measure on $\Omega$.

We consider $F_m$ as a function on $\Omega_B$; that is, we set $G = F_m \circ T$. The goal is to estimate the derivative of $G$, so define the derivative relative to $\omega_{e,j}$ of a function $f:\Omega_B \to \mathbb{R}$ as
\[
\left( \Delta_{e,j} f\right) (\omega) = f(\omega^{e,j,+}) - f(\omega^{e,j,-})\ ,
\]
where $\omega^{e,j,+}$ agrees with $\omega$ except possibly at $\omega_{e,j}$, where it is 1, and $\omega^{e,j,-}$ agrees with $\omega$ except possibly at $\omega_{e,j}$, where it is 0. Then the following analogue of \cite[Eq.~(3)]{benaimrossignol} holds.
\begin{lma} \label{eqn: entropybydelta}
We have the following inequality:
\[
\sum_{k=1}^\infty Ent(V_k^2) \leq \sum_e \sum_{j=1}^\infty \mathbb{E}_\pi (\Delta_{e,j} G)^2\ .
\]
\end{lma}
\begin{proof}
Define a filtration of $\Omega_B$ by enumerating the edges of $\mathcal{E}^d$ as $\{e_1, e_2, \ldots\}$ as before and setting $\mathcal{G}_k$ as the sigma-algebra generated by $\{\omega_{e_r,j} : r \leq k, j \in \mathbb{N}\}$. Also define $W_k = \mathbb{E}_\pi \left[ G \mid \mathcal{G}_k\right]$. It is straightforward to verify that, because $\mathbb{P} = \pi \circ T^{-1}$,
\[
\mathbb{E}[F_m \mid \mathcal{F}_k] (T(\omega_B)) = \mathbb{E}_\pi[G \mid \mathcal{G}_k](\omega_B) \text{ for } \pi\text{-almost every }\omega_B \in \Omega_B\ .
\]
Therefore $Ent(V_k^2) = Ent_\pi(W_k^2)$ for each $k$. Using tensorization of entropy (Theorem~\ref{thm: tensorization}),
\[
\sum_{k=1}^\infty Ent_\pi(W_k^2) \leq \sum_{k=1}^\infty \mathbb{E}_\pi \sum_e \sum_{j=1}^\infty Ent_{\pi_{e,j}}W_k^2\ .
\]
For this to be true, we need to check the condition $W_k^2 \in L^2$, or that $V_k \in L^4$. Since $V_k$ is a difference of martingale sequence terms, it suffices to show that $\tau(0,x)$ is in $L^4$. But this follows from \cite[Lemma~3.1]{coxdurrett}: if $Y = \min\{t_1, \ldots, t_{2d}\}$ is a minimum of $2d$ i.i.d. variables distributed as $t_e$, then $\tau(0,x) \in L^4$ if and only if $\mathbb{E}Y^4<\infty$. By Chebyshev's inequality,
\[
\mathbb{E}Y^\alpha = \alpha \int_0^\infty y^{\alpha-1} \mathbb{P}(t_e > y)^{2d}~\text{d}y \leq C\int_0^\infty y^{\alpha-1-4d}~\text{d}y\ ,
\]
which is finite if $\alpha < 4d$. In particular, since $d \geq 2$, one has $W_k^2 \in L^2$.

Recall the Bonami-Gross inequality \cite{bonami, gross}, which says that if $f:\{0,1\} \to \mathbb{R}$ and $\nu$ is uniform on $\{0,1\}$ then
\[
Ent_\nu f^2 \leq (1/2)(f(1)-f(0))^2\ .
\]
Therefore we get the upper bound $\sum_{j=1}^\infty \sum_e \sum_{k=1}^\infty\mathbb{E}_\pi(\Delta_{e,j}W_k)^2$. For fixed $e,j$,
\[
\Delta_{e,j}W_k = \begin{cases}
0 & \text{ if } k<j \\
\mathbb{E}_\pi[\Delta_{e,j}G \mid \mathcal{G}_k] & \text{ if } k=j \\
\mathbb{E}_\pi[\Delta_{e,j}G \mid \mathcal{G}_k] - \mathbb{E}_\pi[\Delta_{e,j}G \mid \mathcal{G}_{k-1}] & \text{ if } k > j
\end{cases}\ .
\]
The first follows because when $k<j$ then $W_k$ does not depend on $\omega_{e,j}$, as this variable is integrated out. A similar idea works for the second, noting that $\Delta_{e,j} \mathbb{E}_\pi[G \mid \mathcal{G}_{k-1}] = 0$. The third is straightforward. Using orthogonality of martingale differences, $\sum_{k=1}^\infty \mathbb{E}_\pi (\Delta_{e,j}W_k)^2 = \mathbb{E}_\pi(\Delta_{e,j} G)^2$ and this completes the proof.
\end{proof}

%For the bound, we assume three conditions:
%\begin{equation}\label{eq: assumption}
%\limsup_{x \to I^+} (-\log F(x^-) (x-I)^2) <\infty\ ,
%\end{equation}
%\begin{equation}\label{eq: bound_assumption}
%Ent_\mu(t_e^2)  < \infty, \text{ and }
%\end{equation}
%\begin{equation}\label{eq: p_c_bound}
%\mu(\{0\})< p_c(d)\ .
%\end{equation}
%Here we have used the notation $F(x^-) = \mu((-\infty,x))$. 

\subsubsection{Control by edges in geodesics}

The first major step is to bound the sum of discrete derivatives by a weighted average of edge-weights in geodesics. The bound we give is analogous to what would appear if we had a log-Sobolev inequality for $\mu$ (see the approach in Benaim-Rossignol \cite{benaimrossignol}); however, we get a logarithmic singularity as $t_e \downarrow I$.
\begin{prop}\label{prop: intermediate}
There exists $\mathbf{C}_{28}$ such that for all $x$,
\[
\sum_e \sum_{j=1}^\infty \mathbb{E}_\pi (\Delta_{e,j} G)^2 \leq \mathbf{C}_{28} \mathbb{E} \sum_e (1-\log F(t_e)) \mathbf{1}_{\{e \in Geo(0,x)\}}\ .
\]
\end{prop}

\begin{proof}
We begin by using convexity of the square function to get
\begin{equation}\label{eq: first_eq}
\sum_e \sum_{j=1}^\infty \mathbb{E}_\pi \left(\Delta_{e,j} G \right)^2 \leq \frac{1}{\#B_m} \sum_{z \in B_m} \left[ \sum_e \sum_{j=1}^\infty \mathbb{E}_\pi \left( \Delta_{e,j} \tau_z \right)^2 \right]\ ,
\end{equation}
where $\tau_z=\tau(z,z+x)$. Write $\mathbb{E}_{e^c}$ for expectation relative to $\prod_{f \neq e} \pi_f$ and for any $i\geq 1$, let $\pi_{e,\geq i}$ be the measure $\prod_{k \geq i} \pi_{e,k}$. Further, for $j \geq 1$ write 
\[
\omega_B = (\omega_{e^c}, \omega_{e,< j}, \omega_{e,j}, \omega_{e,> j})\ ,
\]
where $\omega_{e^c}$ is the configuration $\omega_B$ projected on the coordinates $(\omega_{f,k} : f \neq e,~ k \geq 1)$, $\omega_{e,<j}$ is $\omega_B$ projected on the coordinates $(\omega_{e,k} : k < j)$ and $\omega_{e, > j}$ is $\omega_B$ projected on the coordinates $(\omega_{e,k} : k > j)$.

The expectation in \eqref{eq: first_eq} is now
\begin{align}
&~\mathbb{E}_{e^c} \mathbb{E}_{\pi_{e,1}} \cdots \mathbb{E}_{\pi_{e,j-1}} \left[ \mathbb{E}_{\pi_{e,\geq j}} \left( \Delta_{e,j} \tau_z \right)^2 \right] \nonumber \\
=& ~\mathbb{E}_{e^c} \left[ \frac{1}{2^{j-1}} \sum_{\sigma \in \{0,1\}^{j-1}} \left[ \mathbb{E}_{\pi_{e,\geq j}} \left( \Delta_{e,j} \tau_z (\omega_{e^c}, \sigma, \omega_{e,j}, \omega_{e, > j}) \right)^2 \right] \right] \label{eq: second_eq}\ ,
\end{align}
and the innermost term is
\begin{equation}\label{eq: sum_on_home}
\mathbb{E}_{\pi_{e, \geq j}} \left( \tau_z(\omega_{e^c},\sigma, 1, \omega_{e,>j}) - \tau_z(\omega_{e^c}, \sigma,0, \omega_{e,>j}) \right)^2 \ .
\end{equation}
%For $j=1$ we have the same form but with no sum or pre-factor:
%\[
%\mathbb{E}_{\pi_e} (\tau_z(\omega_{e^c}, 1,\omega_{e,>1}) - \tau_z(\omega_{e^c}, 0,\omega_{e,>1}))^2\ .
%\] 

Because of Lemma~\ref{lem: deriv}, we can rewrite \eqref{eq: sum_on_home} as
\begin{equation}\label{eq: hopefully_last}
\mathbb{E}_{\pi_{e,\geq j}} \min\left\{ (T_e(\sigma,1,\omega_{e,>j})-T_e(\sigma,0,\omega_{e,>j}))^2, (D_{z,e}-T_e(\sigma,0,\omega_{e,>j}))_+^2\right\}\ .
\end{equation}
Note that this allows us to assume $D_{z,e} > I$:
\begin{equation}\label{eq: D_{z,e}_reduction}
\mathbb{E}_\pi (\Delta_{e,j} \tau_z)^2 = \mathbb{E}_{e^c} \left[ \frac{1}{2^{j-1}} \sum_{\sigma \in \{0,1\}^{j-1}} \left[ \mathbb{E}_{\pi_{e,\geq j}} \left( \Delta_{e,j} \tau_z (\omega_{e^c}, \sigma, \omega_{e,j}, \omega_{e, > j}) \right)^2 \right] \mathbf{1}_{\{I<D_{z,e}\}}\right] \ .
\end{equation}
%and when $j=1$,
%\[
%\mathbb{E}_{\pi_e} \mathbf{1}_{\{T_e(0,\omega_{e,>1}) \leq D_{z,e}\}} \min\{ (T_e(1,\omega_{e,>1}) - T_e(0,\omega_{e,>1}))^2, (D_{z,e}-T_e(0,\omega_{e,>1}))_+^2 \}\ .
%\]
To simplify notation in the case $j \geq 2$, we write the values $a_{1,j-1}, \ldots, a_{2^{j-1}-1,j-1}$ as $a_1, \ldots, a_{2^{j-1}-1}$ and for a fixed $\sigma \in \{0,1\}^{j-1}$, $a_\sigma$ for $a_{i((\sigma,0,\omega_{e,>j}),j-1),j-1}$ (note that this does not depend on the configuration outside of $\sigma$). Also we write $a'_\sigma$ for the element of the partition that follows $a_\sigma$ (when there is one; that is, when $\sigma$ is not $(1, \ldots, 1)$). Last, we abbreviate $T_e(\sigma,c,\omega_{e,>j})$ by $T_{e,j}(\sigma,c)$ for $c=0,1$. With this notation, we claim the inequalities
\[
a_\sigma \leq T_{e,j}(\sigma,0) \leq T_{e,j}(\sigma,1) \leq a'_\sigma \text{ when } \sigma \neq (1, \ldots, 1) \text{ and } j \geq 2\ .
\]
The first and third inequalities follow from the nesting part of Lemma~\ref{lem: encoding}. The second holds because of the monotonicity part. Therefore we can give an upper bound for \eqref{eq: hopefully_last} when $j \geq 2$ of
\[
\begin{cases}
0 & \text{ if } D_{z,e} \leq a_\sigma \\
\mathbb{E}_{\pi_{e,\geq j}} \min\{D_{z,e}-a_\sigma, T_{e,j}(\sigma,1) - a_\sigma\}^2  \mathbf{1}_{\{T_{e,j}(\sigma,0) < D_{z,e}\}}  & \text{ if } \sigma \neq (1, \ldots, 1) \text{ and } a_\sigma < D_{z,e} \leq a'_\sigma\\
&\text{ or } \sigma = (1, \ldots, 1)  \\
(a'_\sigma - a_\sigma )^2 & \text{ if } a'_\sigma \leq D_{z,e} 
\end{cases}\ .
\]
%In the case $j=1$, the upper bound is simply $(D_{z,e} - a_0)_+^2 \pi_e (T_e(0,\omega_{e,>1}) < D_{z,e})$. 
(Here and above we have strict inequality in the condition of the indicator function since when $T_e(\sigma,0,\omega_{e,>j})=D_{z,e}$, \eqref{eq: hopefully_last} is zero.) With this, when $j \geq 2$, the integrand of $\mathbb{E}_{e^c}$ in \eqref{eq: D_{z,e}_reduction} is no bigger than
\begin{align}
\frac{1}{2^{j-1}} &\bigg[ (a_1-a_0)^2 + \cdots + (a_s-a_{s-1})^2 \nonumber  \\
&+ \mathbb{E}_{\pi_{e,\geq j}} \min\{D_{z,e}-a_s,T_{e,j}(\sigma(D_{z,e}),1)-a_s\}^2 \mathbf{1}_{\{T_{e,j}(\sigma(D_{z,e}),0) < D_{z,e}\}} \bigg] \mathbf{1}_{\{I<D_{z,e}\}} \label{eq: descending}\ .
\end{align}
Here we have written $s$ for the largest index $i$ such that $a_i < D_{z,e}$ and $\sigma(D_{z,e})$ for the configuration such that $a_{\sigma(D_{z,e})} = a_s$. In the case $j=1$, we have the similar upper bound
\begin{equation}\label{eq: j_one}
\mathbb{E}_{\pi_{e,\geq j}} \min\{ D_{z,e}-I, T_{e,1}(1)-I\}^2 \mathbf{1}_{\{T_{e,1}(0) < D_{z,e}\}} \mathbf{1}_{\{I<D_{z,e}\}}\ .
\end{equation}
Either way, writing $\vec{1}_j$ (respectively $\vec{0}_j$) for the configuration $(1, \ldots, 1)$ (respectively $(0, \ldots, 0)$) of length $j$,
\begin{equation}\label{eq: new_eq_1}
\mathbb{E}_{\pi_e} (\Delta_{e,j} \tau_z)^2 \leq \frac{1}{2^{j-1}} \mathbb{E}_{\pi_{e,\geq j}} \left[ \min\{D_{z,e},T_{e,j}(\vec 1_{j-1},1)\}^2 \mathbf{1}_{\{T_{e,j}(\vec 0_{j-1},0) < D_{z,e}\}}\right] \mathbf{1}_{\{I<D_{z,e}\}}\ .
\end{equation}
Note that $\min\{D_{z,e},T_{e,j}(\vec 1_{j-1},1)^2\}$ is an increasing function of $\omega_{e,\geq j}$ (with all other variables fixed), whereas $\mathbf{1}_{\{T_{e,j}(\vec 0_{j-1},0)< D_{z,e}\}}$ is decreasing (here we use monotonicity of $T_e$). Therefore we can apply the Harris-FKG inequality \cite[Theorem~2.15]{BLM} and sum over $j$ for the upper bound 
\begin{equation}\label{eq: new_eq_2}
\mathbb{E}_{\pi_e} \sum_{j=1}^\infty (\Delta_{e,j} \tau_z)^2 \leq \sum_{j=1}^\infty \frac{1}{2^{j-1}} \left[ \mathbb{E}_{\pi_{e,\geq j}} \min\{D_{z,e},T_{e,j}(\vec 1_{j-1},1)\}^2~ \pi_{e,\geq j}(T_{e,j}(\vec 0_{j-1},0) < D_{z,e}) \right]\mathbf{1}_{\{I<D_{z,e}\}}\ .
\end{equation}

The goal is now to give a useful bound for this sum. To do this, we consider two types of values of $j$. Note that $F(D_{z,e}^-)>0$ and therefore for some $j$, $F(D_{z,e}^-) \geq 2^{-j}$. So define
\[
J(D_{z,e}) = \min \{j \geq 2 : F(D_{z,e}^-) \geq 2^{-(j-1)}\}\ .
\]
Note that
\begin{equation}\label{eq: J_bound}
1-\log_2 F(D_{z,e}^-) \leq J(D_{z,e}) \leq 2-\log_2 F(D_{z,e}^-)\ .
\end{equation}

We will estimate the term $\pi_{e,\geq j}(T_{e,j}(\vec 0_{j-1},0) < D_{z,e})$ only when $j < J(D_{z,e})$. By definition, it is
\[
\left( \prod_{k \geq j} \pi_{e,k}\right) (\{\omega_e : T_e(0,\ldots, 0,\omega_{e,j+1}, \ldots) < D_{z,e}\}) = \pi_e (\{\omega_e : T_e(0, \ldots, 0, \omega_{e,j+1}, \ldots) < D_{z,e}\})\ .
\]
The event in $\Omega_e$ listed on the right depends only on $\omega_{e,k}$ for $k > j$, so it is independent (under $\pi_e$) of the state of the first $j$ coordinates. Thus the above equals
\[
2^j \pi_e (T_e(0, \ldots, 0, \omega_{e,j+1}, \ldots) < D_{z,e},~ \omega_{e,1}, \ldots, \omega_{e,j} = 0) \leq 2^j \pi(T_e(\omega_e) < D_{z,e}) = 2^j F(D_{z,e}^-)\ .
\]
Using this inequality for $j <J(D_{z,e})$, \eqref{eq: new_eq_2} becomes
\begin{align}
\mathbb{E}_{\pi_e} \sum_{j=1}^\infty (\Delta_{e,j} \tau_z)^2 &\leq 2F(D_{z,e}^-) \mathbb{E}_{\pi_{e,\geq 1}}T_{e,1}(1)^2 \mathbf{1}_{\{I<D_{e,z}\}} + 2F(D_{z,e}^-) \sum_{j=2}^{J(D_{z,e})-1} D_{z,e}^2 \mathbf{1}_{\{I<D_{z,e}\}} \label{eq: new_eq_3a}\\
&+ \sum_{j=J(D_{z,e})}^\infty \frac{1}{2^{j-1}} \left[ \mathbb{E}_{\pi_{e,\geq j}} \min\{D_{z,e},T_{e,j}(\vec 1_{j-1},1)\}^2 \right]\mathbf{1}_{\{I<D_{z,e}\}} \label{eq: new_eq_3b}\ .
\end{align}
The second term on the right of \eqref{eq: new_eq_3a} is bounded by noting that when this sum is nonempty (that is, $J(D_{z,e})>2$), it follows that $F(D_{z,e}^-) <1/2$ and so $D_{z,e} \leq a_{1,1}$. Using this with \eqref{eq: J_bound} we obtain
\begin{equation}\label{eq: mid_bound}
 2F(D_{z,e}^-) \sum_{j=2}^{J(D_{z,e})-1} D_{z,e}^2 \mathbf{1}_{\{I<D_{z,e}\}} \leq 2F(D_{z,e}^-)(1-\log_2 F(D_{z,e}^-)) a_{1,1}^2 \mathbf{1}_{\{I<D_{z,e}\}}\ .
\end{equation}

We next bound $\mathbb{E}_{\pi_{e,\geq j}} T_{e,j}(\vec 1_{j-1},1)^2$. Because $T_{e,j}(\vec 1_{j-1},1)$ only depends on $\omega_e$ through $\omega_{e,>j}$,
\[
\mathbb{E}_{\pi_e} T_{e,j}(\vec 1_{j-1},1)^2 = 2^j \mathbb{E}_{\pi_e} T_{e,j}(\vec 1_{j-1},1)^2 \mathbf{1}_{\{\omega_{e,\leq j} = \vec 1_j\}} = 2^j \mathbb{E}_{\pi_e} T_e^2 \mathbf{1}_{\{\omega_{e,\leq j} = \vec 1_j\}}\ .
\]
%However if $\omega_{e,\leq j} = \vec 1_j$ then $T_e(\omega_e) \geq a_{2^j-1,j}$, so this is bounded by $2^j \mathbb{E}_{\pi_e} T_e(\omega_e)^2 \mathbf{1}_{\{T_e(\omega_e) \geq a_{2^j-1,j}\}}$. Using $\mu = \pi_e \circ T_e^{-1}$, 
%\[
%\mathbb{E}_{\pi_e} T_{e,j}(\vec 1_{j-1},1)^2 \leq 2^j \mathbb{E}_\mu t_e^2 \mathbf{1}_{\{t_e \geq a_{2^j-1,j}\}}\ .
%\]
Thus in \eqref{eq: new_eq_3a},
\begin{equation}\label{eq: beginning_bound}
2F(D_{z,e}^-) \mathbb{E}_{\pi_{e,\geq 1}} T_{e,1}(1)^2 \mathbf{1}_{\{I<D_{z,e}\}} \leq 4F(D_{z,e}^-) \mathbb{E}_\mu t_e^2 \mathbf{1}_{\{I<D_{z,e}\}}
\end{equation}
and
\[
\eqref{eq: new_eq_3b} \leq 2\sum_{j=J(D_{z,e})}^\infty \left[ \mathbb{E}_{\pi_e} \min\{D_{z,e}, T_e\}^2 \mathbf{1}_{\{\omega_{e,\leq j} = \vec 1_j\}}\right] \mathbf{1}_{\{I < D_{z,e}\}}
\]
We now consider two cases. If $D_{z,e} \leq a_{1,1}$ then we use \eqref{eq: J_bound} to obtain the upper bound
\begin{align*}
\eqref{eq: new_eq_3b} \leq 2a_{1,1}^2 \sum_{j=J(D_{z,e})}^\infty \pi_e(\omega_{e,\leq j} = \vec 1_j) \mathbf{1}_{\{I < D_{z,e}\}} &= 2a_{1,1}^2 \sum_{j=J(D_{z,e})}^\infty 2^{-j} \mathbf{1}_{\{I<D_{z,e}\}} \\
&\leq 4a_{1,1}^2 2^{-J(D_{z,e})} \mathbf{1}_{\{I<D_{z,e}\}} \\
&\leq 2 a_{1,1}^2 F(D_{z,e}^-) \mathbf{1}_{\{I<D_{z,e}\}}\ .
\end{align*}
On the other hand, if $D_{z,e}>a_{1,1}$ then we use the bound
\[
\eqref{eq: new_eq_3b} \leq 2 \left[\mathbb{E}_{\pi_e} T_e^2 N\right] \mathbf{1}_{\{I<D_{z,e}\}}, \text{ where } N = \max\{j \geq 1 : \omega_{e,\leq j} = \vec 1_j\}\ .
\]
This is bounded by the variational characterization of entropy, Proposition~\ref{prop: characterization}. The expectation is no larger than 
\[
2~Ent_\mu t_e^2 + 2 \mathbb{E}_\mu t_e^2 \log \mathbb{E}_{\pi_e} e^{N/2}\ .
\]
Because $N$ has a geometric distribution, this is bounded by $\mathbf{C}_{29}$ independently of $e$. As $D_{z,e}>a_{1,1}$, one has $F(D_{z,e}^-) \geq 1/2$ and so we obtain
\[
\eqref{eq: new_eq_3b} \leq 4\mathbf{C}_{29}F(D_{z,e}^-) \mathbf{1}_{\{I<D_{z,e}\}}\ .
\]
Combined with the case $D_{z,e} \leq a_{1,1}$, our final bound is
\begin{equation}\label{eq: real_end_bound}
\eqref{eq: new_eq_3b} \leq (4\mathbf{C}_{29}+2a_{1,1}^2) F(D_{z,e}^-) \mathbf{1}_{\{I<D_{z,e}\}}\ .
\end{equation}

Putting together the pieces, \eqref{eq: beginning_bound} with \eqref{eq: mid_bound} and \eqref{eq: real_end_bound},
\begin{equation}\label{eq: lasagna}
\mathbb{E}_{\pi_e}\sum_{j=1}^\infty (\Delta_{e,j} \tau_z)^2 \leq \mathbf{C}_{30} F(D_{z,e}^-)\mathbf{1}_{\{I<D_{z,e}\}} - \mathbf{C}_{31} F(D_{z,e}^-) \log F(D_{z,e}^-) \mathbf{1}_{\{I<D_{z,e}\}}\ .
\end{equation}
To bound terms of the second form we use a lemma.
\begin{lma}\label{lem: ibp}
For any $y>I$, we have
\begin{equation}
 -F(y^-)\log F(y^-)\le -\int_{[I,y)}\log F(a)\,\mu(\mathrm{d}a)\ .\label{eqn: ibpbound}
\end{equation}
\begin{proof}
Let $\epsilon>0$. The function $\log F(x)$ is increasing on $(I, \infty)$. The usual Lebesgue construction gives a measure $\nu$ on $(I,\infty)$ such that
\[\nu(a,b] = \log F(b)-\log F(a)\ge 0\]
for $a,b\in (I,\infty)$.
Fix $x\in (I+\epsilon,\infty)$, and consider the square 
\[\square= (I+\epsilon, x]\times (I+\epsilon, x]\ .\]
It has two parts:
\begin{gather}
\{(a,b):I+\epsilon < a<b \le x\},\label{eqn: one}\\
\{(a,b):I+\epsilon < b\le a \le x\} \label{eqn: two}.
\end{gather}
Thus,
\[(\mu\times \nu)(\square) = \iint_\eqref{eqn: one} (\mu\times\nu)(\mathrm{d}a\mathrm{d}b)+ \iint_\eqref{eqn: two} (\mu\times\nu)(\mathrm{d}a\mathrm{d}b).\]
By Fubini's theorem, the double integrals may be computed as iterated integrals
\begin{align}
\iint_\eqref{eqn: one} (\mu\times\nu)(\mathrm{d}a\mathrm{d}b)&=\int_{(I+\epsilon,x]} \mu((I+\epsilon,b))\nu(\mathrm{d}b)= \int_{(I+\epsilon,x]}(F(b^-)-F(I+\epsilon))\log F(\mathrm{d}b)\label{eqn: int1}\\
\iint_\eqref{eqn: two} (\mu\times\nu)(\mathrm{d}a\mathrm{d}b)&=\int_{(I+\epsilon,x]} \nu((I+\epsilon,a])\mu(\mathrm{d}a)= \int_{(I+\epsilon,x]}(\log F(a)-\log F(I+\epsilon)) F(\mathrm{d}a) \label{eqn: int2}.
\end{align}
By definition of the product measure,
\[(\mu\times \nu)(\square) = (F(x)-F(I+\epsilon))\cdot(\log F(x)-\log F(I+\epsilon)).\]
This gives the equality:
\begin{align*}
(F(x)-F(I+\epsilon))\cdot(\log F(x)-\log F(I+\epsilon)) &=  \int_{(I+\epsilon,x]}F(b^-)\log F(\mathrm{d}b)  +\int_{(I+\epsilon,x]}\log F(a)F(\mathrm{d}a) \\
&\quad -F(I+\epsilon)(\log F(x)-\log F(I+\epsilon)) \\
&\quad -\log F(I+\epsilon) (F(x)-F(I+\epsilon)).
\end{align*}
After performing cancellations, we obtain
\begin{equation} \label{eqn: F}
F(x)\log F(x)-F(I+\epsilon)\log F(I+\epsilon) =  \int_{(I+\epsilon,x]}F(b^-)\log F(\mathrm{d}b)+\int_{(I+\epsilon,x]}\log F(a)F(\mathrm{d}a).
\end{equation}
Since $F(b^-)\ge 0$, this implies the estimate
\[-\int_{(I+\epsilon,x]} \log F(a)\,\mu(\mathrm{d}a)-F(I+\epsilon)\log F(I+\epsilon) \ge -F(x)\log F(x).\]
Taking $\epsilon \downarrow 0$ and using the right continuity of $F$,
\[-\int_{(I,x]}\log F(a)\,\mu(\mathrm{d}a) -F(I)\log F(I) \ge -F(x)\log F(x),\]
where the second term is interpreted as $0$ if $F(I)=0$. Since $F(I)=\mu(\{I\})$,
\[
-F(x)\log F(x) \leq - \int_{[I,x]} \log F(a) \mu(\text{d}a)\ .
\]
Taking $x \uparrow y$, \eqref{eqn: ibpbound} is proved.
\end{proof}
\end{lma}

Apply the last lemma in \eqref{eq: lasagna} with $y=D_{z,e}$:
\begin{align*}
\sum_e \mathbb{E}_\pi \sum_{j=1}^\infty (\Delta_{e,j} \tau_z)^2 &\leq \mathbf{C}_{32} \sum_e \mathbb{E}_{e^c} \int_{[I,D_{z,e})} (1-\log F(a))~\mu(\text{d}a) \\
&= \mathbf{C}_{32} \mathbb{E} \sum_e (1-\log F(t_e)) \mathbf{1}_{\{I\leq t_e < D_{z,e}\}}\ .
\end{align*}
By Lemma~\ref{lem: deriv}, if $t_e < D_{z,e}$ then $e$ is in $Geo(z,z+x)$, so this is bounded above by
\[
\mathbf{C}_{32} \mathbb{E} \sum_e (1- \log F(t_e)) \mathbf{1}_{\{e \in Geo(z,z+x)\}}\ .
\]
Translating back from $z$ to $0$ and putting this in \eqref{eq: first_eq} proves the proposition.
\end{proof}

\subsubsection{Lattice animals}
To bound the right side of the inequality in Proposition~\ref{prop: intermediate} we need finer control than what is given by Kesten's geodesic length estimates, due to the possible singularity of $\log F(t_e)$ as $t_e \downarrow I$. The idea will be that very few edges $e$ on a geodesic have $t_e$ close to $I$. To bound the precise number, we give the main result of this section:
\begin{thm}\label{thm: low_density}
Assume \eqref{eqn: geodesics} and $\mathbb{E} Y^\alpha<\infty$ for some $\alpha>1$, where $Y$ is the minimum of $2d$ i.i.d. variables distributed as $t_e$. There exists $\mathbf{C}_{33}$ such that for all $x \in \mathbb{Z}^d$ and any Borel set $B \subset \mathbb{R}$,
\[
\mathbb{E} \#\{e \in Geo(0,x) : t_e \in B\} \leq \mathbf{C}_{33}\|x\|_1 \mu(B)^{\frac{\alpha-1}{\alpha d}}\ .
\]
\end{thm}

The proof will require an excursion into the theory of greedy lattice animals. We say that a finite set of vertices $\alpha \subseteq \mathbb{Z}^d$ is a lattice animal if it is connected (under graph connectedness on $\mathbb{Z}^d$). One fundamental result on lattice animals is the following, taken from \cite[Lemma 1]{coxetal}, which describes how a lattice animal may be covered by boxes. We set the notation $B(l) = [-l,l]^d.$

\begin{lma}[Cox-Gandolfi-Griffin-Kesten]
\label{thm:coxetal_cover}
Let $\alpha$ be a lattice animal with $0 \in \alpha$ and $\# \alpha = n,$ and let $1 \leq l \leq n.$
There exists a sequence $x_0, x_1, \ldots, x_r \in \mathbb{Z}^d$, where $r = \lfloor 2n / l \rfloor$, such that $x_0 = 0$,
\begin{equation}
\label{eq:alpha_cover}
\alpha \subseteq \bigcup_{i=0}^r (lx_i + B(2l)),
\end{equation}
and
\[\|x_{i+1} - x_i\|_{\infty} \leq 1, \quad 0\leq i \leq r-1. \]
\end{lma}

We will use the above theorem in a similar setting to the original model for which it is proved.
Let $\Xi_n$ denote the set of all self-avoiding paths
\[\gamma = (0 = v_0, e_1, v_1, \ldots, e_n, v_n) \]
which begin at the origin and contain $n$ edges.
Denote by $V(\gamma)$ the vertex set of $\gamma.$  Assume that we have an edge-indexed set of i.i.d. variables $\{X_e\}_e$, where 
$X_e = 1$ with probability $p$ and $0$ otherwise, and denote the joint distribution of $\{X_e\}$ by $\mathbb{P}_p$; we denote expectation under this measure by $\mathbb{E}_p.$ If $\gamma \in \Xi_ n$ for some $n$, we define
$X(\gamma) = \sum_{e \in \gamma} X_e.$ Last, let
\begin{equation}\label{eq: N_def}
N_n := \max_{\gamma \in \Xi_n} X(\gamma). 
\end{equation}

The following lemma (and the proof thereof) is an adaptation of J. Martin's \cite[Prop. 2.2]{martin} extension of a theorem of S. Lee \cite{lee}.
\begin{lma}\label{lem: lee}
There is a constant $C_d < \infty$ depending only on the dimension $d$ such that, for all $p \in (0,1]$ and all $n \in \mathbb{N},$
\begin{equation}
\label{eq:pd_scaling}
\frac{\mathbb{E}_p N_n}{n p^{1/d}} < C_d.
\end{equation}
\end{lma}
\begin{proof}
Let $p\in (0,1]$ be arbitrary. We first consider the case that $np^{1/d} \leq 1.$ In this case, we have
\[
\frac{\mathbb{E}_p N_n}{n p^{1/d}} \leq \frac{1}{n p^{1/d}}\sum_{e \in [-n,n]^d} \mathbb{E}_pX_e \leq  \frac{2d(2 n + 1)^d p}{n p^{1/d}} \leq 2d(3^d) (p^{1/d}n)^{d-1} \leq 3^{d+1}d.
\]
In the case that $n p^{1/d} > 1,$ we set $ l = \lceil p^{-1/d} \rceil$. Note that for any $\gamma \in \Xi_n,$ $V(\gamma)$ is a lattice animal with $n + 1$ vertices. In particular, it can by covered using the results of Theorem \ref{thm:coxetal_cover}.
So for any $s \geq 0,$ 
\begin{align}
\mathbb{P}_p \left(\frac{N_n}{n p^{1/d}} \geq s\right) = \mathbb{P}_p \left(\max_{\gamma \in \Xi_n} X(\gamma) \geq n p^{1/d} s \right) &\leq \mathbb{P}_p \left( \max_{x_0, \ldots, x_r} \sum_{\substack{e = \{ x, y \} \\ x,y \in \cup_{i=0}^r (lx_i + B(2l))}} X_e \geq n p^{1/d} s \right)\nonumber\\
\label{eq:la_sum_bd}
&\leq \sum_{x_0, \ldots, x_r} \mathbb{P}_p \left( \sum_{\substack{e = \{ x, y \} \\ x,y \in \cup_{i=0}^r (lx_i + B(2l))}} X_e \geq n p^{1/d} s \right),
\end{align}
where the outer sum is over all connected subsets of $\mathbb{Z}^d$  of cardinality $r+1 = 1 + \lfloor 2(n+1)/l \rfloor \leq 5 n p^{1/d}$ which contain the origin.

The expression in \eqref{eq:la_sum_bd} is bounded above by
\begin{align}
\sum_{x_0, \ldots, x_r} \exp(-np^{1/d} s) \mathbb{E}_p& \exp\left( \sum_{\substack{e = \{ x, y \} \\ x,y \in \cup_{i=0}^r (lx_i + B(2l))}} X_e \right)\nonumber\\
\label{eq:la_prod_bd}
&\leq \sum_{x_0, \ldots, x_r} \exp(-np^{1/d} s) \left[\mathbb{E}_p \exp(X_e)\right]^{\# \{e = \{ x, y \}: x,y \in \cup_{i=0}^r (lx_i + B(2l))\}}.
\end{align}
Now, note that
\begin{itemize}
\item $\mathbb{E}_p \exp(X_e) = 1 - p + p\mathrm{e};$
\item The number of vertices in $B(2l)$ is $(4l + 1)^d,$ so
\[
\# \{e = \{ x, y \} : x,y \in \cup_0^r (lx_i + B(2l))\} \leq (r+1) (2d)(4l+1)^d  \leq \mathbf{C}_{34}(d) n p^{1/d-1}\ .
\]
\item The number of terms in the sum \eqref{eq:la_prod_bd} is at most $3^{d(r+1)} \leq 3^{5dnp^{1/d}}.$
\end{itemize}

Putting the above into \eqref{eq:la_prod_bd}, we have
\begin{align}
\mathbb{P}_p\left(\frac{N_n}{n p^{1/d}} \geq s\right) &\leq \exp(-n p^{1/d} s) 3^{5 d n p^{1/d}} \left[ 1 - p + p \mathrm{e}\right]^{\mathbf{C}_{34}(d) n p^{1/d - 1}}\nonumber\\
&= \exp(-n p^{1/d} s) 3^{5 d n p^{1/d}} \left(\left[ 1 - p + p \mathrm{e}\right]^{1/p}\right)^{\mathbf{C}_{34}(d) n p^{1/d}}\nonumber\\
&\leq \exp(-n p^{1/d} s) 3^{5 d n p^{1/d}} \left[ \mathrm{e}^{\mathrm{e}-1}\right]^{\mathbf{C}_{34}(d) n p^{1/d }}\nonumber\\
&=: \exp(-np^{1/d} s + \mathbf{C}_{35} n p^{1/d}) \label{eq:eminusone},
\end{align}
where $\mathbf{C}_{35} = \mathbf{C}_{35}(d)$ again does not depend on $p$ or $n$. Then we have, for $np^{1/d} > 1,$
\begin{align*}
\mathbb{E}_p\left(\frac{N_n}{n p^{1/d}}\right) \leq \mathbf{C}_{35} + \mathbb{E}_p \left[ \frac{N_n}{n p^{1/d}} - \mathbf{C}_{35} \right]_+ &= \mathbf{C}_{35} + \int_{\mathbf{C}_{35}}^{\infty} \mathbb{P}_p\left(\frac{N_n}{np^{1/d}} \geq s \right) \mathrm{d} s\\
&\leq \mathbf{C}_{35} + \int_{\mathbf{C}_{35}}^{\infty} \exp\left( - np^{1/d}(s-\mathbf{C}_{35})\right) \mathrm{d} s\\
&\leq \mathbf{C}_{35} + \int_{\mathbf{C}_{35}}^{\infty} \exp\left(-(s-\mathbf{C}_{35})\right)\mathrm{d} s \leq \mathbf{C}_{36}
\end{align*}
for some $\mathbf{C}_{36} = \mathbf{C}_{36}(d).$
\end{proof}

We are now ready to prove the theorem.
%The last ingredient we need is a simple bound on the fourth moment of passage times.
%\begin{lma}
%There exists a $C_4$ depending only on $d$ and $\mu$ such that
%\[ \mathbb{E}\tau(0,x)^4 \leq C_4 \|x\|_1^4\]
%for all $x \in \mathbb{Z}^d$. \qed
%\end{lma}
%\begin{proof}
%Under our assumption of $\mathbb{E} t_e^2<\infty$, one can verify that if $Y$ is the minimum of $2d$ i.i.d. variables distributed as $t_e$, then $\mathbb{E} Y^4 < \infty$. Then \cite[Lemma~3.1]{coxdurrett} gives $\mathbb{E} \tau(0,x)^4 < \infty$. To get the advertised bound, use subadditivity of the passage time along a deterministic path from $0$ to $x$ of length $\|x\|_1$.
%\end{proof}
%Let $a > I$, and define 
%\[Y_a(x,y) = \# \{e \in Geo(x,y): t_e < a\}. \]
%\begin{lma}\label{lem: low_density}
%There exists a constant $C$ such that, for all $x \in \mathbb{Z}^d$ and $a > I,$
%\[ \mathbb{E} Y_a(0,x) \leq C \|x\|_1 \left[\mu([I,a)) \right]^{3/(4d)}.\]
%\end{lma}
\begin{proof}[Proof of Theorem~\ref{thm: low_density}]
Consider any deterministic ordering of all finite self-avoiding lattice paths and denote by $\pi(x,y)$ the first geodesic from $x$ to $y$ in this ordering. Writing $Y_B(0,x)$ for the number of edges in $\pi(x,y)$ with weight in $B$, note that it suffices to give the bound for $\mathbb{E} Y_B(0,x)$. Define a set of edge weights $X_e$ as a function of $t_e$:
\[ X_e = \begin{cases}
1 & \text{ if $t_e \in B$ }\\
0 & \text{ otherwise}
\end{cases}\]
and build the random variables $N_n$ for these weights as in \eqref{eq: N_def}.

On the event $\{\# \pi(0,x) \leq i \},$ we have $Y_B(0,x) \leq N_i$. Therefore, for all $x \in \mathbb{Z}^d$ and $\kappa \in \mathbb{N},$
\begin{align*}
\mathbb{E} Y_B(0,x)  &\leq \mathbb{E} N_{\kappa \|x\|_1} + \mathbb{E}\left[\# \pi(0,x) \mathbf{1}_{\{\#\pi(0,x)>\kappa \|x\|_1\}}\right] \\
&= \mathbb{E} N_{\kappa \|x\|_1} +  \int_{\kappa\|x\|_1}^\infty \mathbb{P}\left(\#\pi(0,x) >  s\right) \mathrm{d} s\\
&\leq C_d \kappa \|x\|_1 \mu(B)^{1/d} +  \int_{\kappa\|x\|_1}^\infty \mathbb{P}\left(\#\pi(0,x) >  s\right) \mathrm{d} s.
\end{align*}

To bound the integral above, we use the technique of Kesten (see Eq.~(2.26)-(2.27) in \cite{kestenspeed}). For $b, j > 0,$ denote by $D(j,b,x)$ the event that there exists a self-avoiding path $r$ starting at the origin of at least $j \|x\|_1$ steps but $\tau(r) < j b \|x\|_1.$ Then for any $b > 0,$
\begin{align}
\label{eq:boundlength}
\mathbb{P}\left(\#\pi(0,x) >  j\|x\|_1\right) &\leq \mathbb{P}\left(\tau(0,x) \geq bj\|x\|_1\right) + \mathbb{P}(D(j,b,x)).
\end{align}
By our assumption $\mathbb{E} Y^\alpha<\infty$, \cite[Lemma~3.1]{coxdurrett} implies that there exists $\mathbf{C}_{37}$ such that for all $x$, $\mathbb{E} \tau(0,x)^\alpha \leq \mathbf{C}_{37} \|x\|_1^\alpha$. 
%So 
%\[
%(b j \|x\|_1)^\alpha \mathbb{P}\left(\tau(0,x) \geq b j \|x\|_1\right) \leq \mathbb{E}\tau(0,x)^\alpha \leq  \mathbf{C}_{37} \|x\|_1^\alpha%\ .
%\]
Thus for arbitrary $x \in \mathbb{Z}^d,$
\[\mathbb{P}\left(\tau(0,x) \geq b j \|x\|_1\right) \leq \mathbf{C}_{37} / (b j)^\alpha.\]

Due to assumption \eqref{eqn: geodesics}, we may use Theorem~\ref{thm: kesten_exp} to see that, for $b$ smaller than some $b_0>0$ (which depends on $d$ and $\mu$), the probability of $D(j,b,x)$ is bounded above uniformly in $j$ and $x$ by $\exp (-\mathbf{C}_{38} j \|x\|_1)$.
Inserting this bound into \eqref{eq:boundlength}, we see that for $b$ small enough,
\[\mathbb{P}\left(\#\pi(0,x) >  j\|x\|_1\right) \leq \frac{\mathbf{C}_{37}}{(b j)^\alpha} + \exp(- \mathbf{C}_{38} j\|x\|_1).\]
In particular, setting $r = s/\|x\|_1,$
\begin{align}
\mathbb{E} Y_B(0,x) &\leq C_d \kappa \|x\|_1 \mu(B)^{1/d} + \|x\|_1\int_{\kappa}^{\infty}\left( \frac{\mathbf{C}_{37}}{(br)^\alpha} + \exp(- \mathbf{C}_{38} r \|x\|_1)\right)\mathrm{d} r\nonumber\\
&\leq C_d \kappa \|x\|_1 \mu(B)^{1/d} + \frac{\mathbf{C}_{39} \|x\|_1}{\kappa^{\alpha-1}}\nonumber
\end{align}
for some constant $\mathbf{C}_{39}.$ Choosing $\kappa = \lceil \mu(B)^{-1/(\alpha d)}\rceil$ completes the proof.
\end{proof}

%. Applying this choice to \eqref{eq:kappachoice}, we see
%\begin{align*}
%\mathbb{E} Y_B(0,x) &\leq C_d \lceil\mu(B)^{-1/(4d)}\rceil \|x\|_1 \mu(B)^{1/d} + \frac{C_7 \|x\|_1}{\lceil\mu(B)^{-3/(4d)}\rceil}\\
%&\leq  C_d \left(1+\mu(B)^{-1/(4d)}\right) \|x\|_1 \mu(B)^{1/d} + C_7 \mu(B)^{-3/(4d)} \|x\|_1\\
%&\leq \left( 2 C_d + C_7\right)  \mu(B)^{3/(4d)} \|x\|_1.
%\end{align*}

%Using these three bounds with \eqref{eq: D_{z,e}_reduction} yields 
%\[
%\mathbb{E} \sum_{j=1}^\infty (\Delta_{e,j} \tau_z)^2 \leq \mathbb{E}_{e^c} (6M^2 + 2C_2 - 2M^2 \log_2 F((I+\delta)^-)  + 2C_3) F(D_{z,e}^-) = C_4 \mathbb{E}_{e^c} F(D_{z,e}^-)\ .
%\]
%However, using the definition of $D_{z,e}$, we have
%\[
%\mathbb{E}_{e^c} F(D_{z,e}^-) = \mathbb{E}_{e^c} \mathbb{E}_e \mathbf{1}_{\{t_e < D_{z,e}\}} = \mathbb{P}(e \in Geo(z,z+x))\ ,
%\]
%where $Geo(z,z+x)$ is the intersection of all geodesics from $z$ to $z+x$. Therefore \eqref{eq: first_eq} is no larger than 
%\[
%\frac{C_4}{\# B} \sum_{z \in B} \mathbb{E} \# \{e: e \in Geo(z,z+x)\}\ .
%\]
%By \cite[(2.25)]{kestenspeed}, under assumptions \eqref{eqn: 2log} and \eqref{eqn: geodesics}, there is a $C_5$ such that for all $z$, $\mathbb{E} \#\{e : e \in Geo(z,z+x)\} \leq C_5\|x\|_1$. So we finish with the bound $C_2C_5 \|x\|_1$.

\subsubsection{Finishing the proof of Theorem~\ref{thm: derivative_bound}}\label{sec: finishing_deriv}
We use Theorem~\ref{thm: low_density}, with a dyadic partition of $[I,\infty)$: let
\[
x_0 = \infty \text{ and } x_n = \min \{x : F(x) \geq 2^{-n}\} \text{ for } n \in \mathbb{N}\ .
\]
Note that for any edge $e$, $t_e$ almost surely lies in one of the intervals $[x_i,x_{i-1})$ for $i \geq 1$. This is clear if $I<t_e$. Otherwise we must have $\mu(\{I\})>0$ and we simply take $i$ to be minimal such that $2^{-i} \leq \mu(\{I\})$.

Now the right side of the inequality in Proposition~\ref{prop: intermediate} can be rewritten as 
\begin{align*}
\mathbf{C}_{28} \sum_{i=1}^\infty \sum_e \mathbb{E} &\left[ (1-\log F(t_e)) \mathbf{1}_{\{e \in Geo(z,z+x)\}} \mathbf{1}_{\{t_e \in [x_i,x_{i-1})\}}\right] \\
&\leq \mathbf{C}_{28} \sum_{i=1}^\infty (1-\log F(x_i)) \mathbb{E} \#\{e \in Geo(z,z+x) : t_e \in [I,x_{i-1})\}\ .
\end{align*}
By Theorem~\ref{thm: low_density} with $\alpha=2$, this is bounded by
\[
\mathbf{C}_{28}\mathbf{C}_{33} \|x\|_1 \sum_{i=1}^\infty (1-\log F(x_i)) F(x_{i-1}^-)^{1/(2d)} \leq \mathbf{C}_{28}\mathbf{C}_{33} \|x\|_1 \sum_{i=1}^\infty \frac{1+i}{2^{(i-1)/(2d)}} \leq \mathbf{C}_{27} \|x\|_1\ .
\]

\subsection{Proof of Theorem~\ref{thm: sub linear}}\label{sec: proof}

For $x \in \mathbb{Z}^d$, if $\Var F_m \leq \|x\|_1^{7/8}$ then by Proposition~\ref{prop: approximating}, we are done. Otherwise, under assumptions \eqref{eqn: 2log} and \eqref{eqn: geodesics} we can use \eqref{eq: breakdown} to find for some $\mathbf{C}_3$
\[
\Var~\tau(0,x) \leq \mathbf{C}_3\|x\|_1^{3/4} + \left[ \log \left[ \|x\|_1^{1/8}\right] \right]^{-1} \sum_{k=1}^\infty Ent(V_k^2)\ .
\]
By Theorem~\ref{thm: derivative_bound}, $\sum_{k=1}^\infty Ent(V_k^2) \leq \mathbf{C}_{27}\|x\|_1$, so
\[
\Var~\tau(0,x) \leq \mathbf{C}_3\|x\|_1^{3/4} + \frac{8\mathbf{C}_{27} \|x\|_1}{\log \|x\|_1} \leq \frac{\mathbf{C}_{40} \|x\|_1}{\log \|x\|_1}\ .
\]

\bigskip
\noindent
{\bf Acknowledgements}. We thank S. Sodin for helpful conversations and for pointing out the use of geometric averaging in his paper. We also thank him for careful reading of a previous draft. We are grateful to T. Sepp\"al\"ainen for pointing out an error in the entropy section and A. Auffinger for finding various typos. Last, we thank an anonymous referee for comments that led to a better organized presentation.

\bigskip
\noindent
{\tt Michael Damron: mdamron@math.princeton.edu}

\medskip
\noindent
{\tt Jack Hanson: jthanson@princeton.edu}

\medskip
\noindent
{\tt Philippe Sosoe: psosoe@math.princeton.edu}

\bigskip
\noindent
\begin{tabular}{ll}
\begin{tabular}{l}
Department of Mathematics, Princeton University \\
Fine Hall, Washington Rd. \\
Princeton, NJ 08544
\end{tabular}
&
\begin{tabular}{l}
Department of Physics, Princeton University \\
Jadwin Hall, Washington Rd. \\
Princeton, NJ 08544
\end{tabular}
\end{tabular}

\begin{thebibliography}{1}

\bibitem{alexander} Alexander, K. S. \emph{A note on some rates of convergence in first-passage percolation}, Ann. Appl. Probab. \textbf{3}, 1993.

\bibitem{AZ} Alexander, K. S., Zygouras, N., \emph{Subgaussian concentration and rates of convergence in directed polymers}, Elect. J. Probab., {\bf 18}, no. 5, 2013.

\bibitem{AD11} Auffinger, A., Damron, M. \emph{Differentiability at the edge of the percolation cone and related results in first-passage percolation}, Probab. Theory Relat. Fields, \textbf{156}, 2013.

\bibitem{BKS} Benjamini, I., Kalai,G., Schramm, O., \emph{First-passage percolation has sublinear distance variance}, Ann. Prob. \textbf{31}, 2003.

\bibitem{BStahn} Blair-Stahn, N., \emph{First passage percolation and competition models}, arXiv:1005.0649.

\bibitem{benaimrossignol} Benaim, M., Rossignol, R., \emph{Exponential concentration for first passage percoaltion through modified Poincar\'e inequalities}, Ann. Inst. Henri Poincar\'e, Prob. Stat. \textbf{3}, 2008.

\bibitem{bonami} Bonami, A., \emph{Etude des coefficients de Fourier des fonctions de $L^p(G)$}, Annales de l'Institut Fourier, \textbf{20}, 1970.

\bibitem{BLM} Boucheron, S., Lugosi, G., Massart, P., \emph{Concentration inequalities: a non asymptotic theory of independence}, Oxford University Press, 2013.

\bibitem{chayes} Chayes, L., \emph{On the critical behavior of the first passage time in $d \geq 3$}, Hel. Phys. Acta {\bf 64}, 1991.

\bibitem{coxdurrett} Cox, J. T., Durrett, R., \emph{Some limit theorems for percolation processes with necessary and sufficient conditions}, Ann. Probab. \textbf{4}, 1981.

\bibitem{coxetal} Cox, J. T., Gandolfi, A., Griffin, P., Kesten, H., \emph{Greedy lattice animals I: Upper bounds}, Ann. Appl. Prob. \textbf{3}, 1993.

\bibitem{DK} Damron, M., Kubota, N., \emph{Gaussian concentration for the lower tail in first-passage percolation under low moments}, arXiv: 1406.3105, 2014.

\bibitem{FS} Falik, D., Samorodnitsky, A., \emph{Edge-isoperimetric inequalities and influences}, Combinatorics, Probability and Computing {\bf 16} 693-712, 2007.

\bibitem{federbusch} Federbusch, P., \emph{A partially alternative derivation of a result of Nelson}, J. Phys \textbf{10}, 1969.

\bibitem{GK} Grimmett, G., Kesten, H., \emph{Percolation since Saint-Flour}, arXiv:1207.0373.

\bibitem{gross} Gross, L. \emph{Logarithmic Sobolev Inequalities}, Amer. J. Math. \textbf{97}, 1975.

\bibitem{Howard} Howard, C. D., \emph{Models of first-passage percolation}, Probability on discrete structures, 125-173, Encyclopaedia Math. Sci., 110, Springer, Berlin, 2004.

\bibitem{johansson} Johansson, K. \emph{Shape fluctuations and random matrices}, Comm. Math. Phys. \textbf{209}, 2000.

\bibitem{kpz} Kardar, K., Parisi, G., Zhang, Y., \emph{Dynamic scaling of growing interfaces}, \emph{Phys. Rev. Lett.}, \textbf{56}, 1986. 

\bibitem{kesten} Kesten, H., \emph{Aspects of first-passage percolation}, \'Ecole d'\'et\'e de probabilit\'es de Saint-Flour, XIV--1984, 125--264, Lecture Notes in Math., 1180, Springer.

\bibitem{kestenspeed} Kesten, H., \emph{On the speed of convergence in first-passage percolation}, Ann. Appl. Probab. {\bf 3} 296-338, 1993.

\bibitem{lee} Lee, S., \emph{The power laws of $M$ and $N$ in greedy lattice animals}, Stoch. Proc. Appl. \textbf{69}, 1997.

\bibitem{ledoux} Ledoux, M. \emph{The Concentration of Measure phenomenon}, AMS Math Surveys and Monographs, 2001.

\bibitem{martin} Martin, J., \emph{Linear growth for greedy lattice animals}, Stoch. Proc. Appl. \textbf{98}, 2002.

\bibitem{newman} Newman, C.M., Piza, M.S.T., \emph{Divergence of shape fluctuations in two dimensions}, Ann. Prob. \textbf{23}, 1995.

\bibitem{PP} Pemantle, R., Peres, Y., \emph{Planar first-passage percolation times are not tight}, In NATO ASI Series C Mathematical and Physical Sciences-Advanced Study Institute, volume 420, p. 261-264. Kluwer Acad. Publ. Dordrecht, 1994.

\bibitem{rhee} Rhee, W.T. \emph{On Rates of Convergence for Common Subsequences and First Passage Time}, Ann. Appl. Probab. \textbf{1}. 1995.

\bibitem{Rossignol} Rossignol, R. \emph{Threshold for monotone symmetric properties through a logarithmic Sobolev inequality}, Ann. Probab. {\bf 35}, 2005
%\bibitem{SC} Saloff-Coste, L., \emph{Lectures on finite Markov chains. Ecole d'Et'e de Probabilit'es de St-Flour 1996}. Lecture Notes in Math. 1665, Springer-Verlag, 1997. 

\bibitem{sodin13} Sodin, S., \emph{Positive temperature versions of two theorems on first-passage percolation}, arXiv: 1301.7470, 2013.

\bibitem{stam} Stam, A. \emph{Some inequalities satisfied by the quantities of information of Fisher and Shannon}, Inform. Control \textbf{2}, 1959.

\bibitem{talagrand} Talagrand, M. \emph{Concentration of measure and isoperimetric inequalities in product spaces}, Publ. Math. I.H.E.S. \textbf{81}, 1995.

\bibitem{talagrand-russo} Talagrand, M. \emph{On Russo's Approximate Zero-One Law}, Ann. Prob. \textbf{22}, 1994.

\bibitem{zhang} Zhang, Y. \emph{Shape Fluctuations Are Different in Different Directions}, Ann. Prob. {\bf 36}, 2008.

\bibitem{zhang2} Zhang, Y. \emph{On the concentration and the convergence rate with a moment condition in first passage percolation}, Stoch. Proc. Appl. \textbf{120}, 2010.

\bibitem{zhang3} Zhang, Y. \emph{Double behavior of critical first-passage percolation}, Perplexing problems in probability, 143-158, Progr. Probab., 44, Birkh\"auser Boston, Boston, MA, 1999.

\end{thebibliography}
\end{document}